\documentclass[11pt]{amsart}
\usepackage{calc,amssymb,amsmath,ulem,stmaryrd}
\usepackage{alltt}
\RequirePackage[dvipsnames,usenames]{color}

\input{kmacros3.sty}
\input{xy}
\xyoption{all}
\usepackage{hyperref}
\usepackage{graphicx}
\usepackage[all,cmtip]{xy}
\usepackage{verbatim}
\usepackage[left=1.25in,top=1.05in,right=1.25in,bottom=1.05in]{geometry}

\usepackage{mabliautoref}

\theoremstyle{theorem}

\renewcommand{\O}{\mathcal O}

\renewcommand{\to}[1][]{\xrightarrow{\ #1\ }}

\newcommand{\Tt}{{\mathfrak{T}}}

\begin{document}
\numberwithin{equation}{theorem}
\title[Test ideals of non-principal ideals]{Test ideals of non-principal ideals: Computations, Jumping Numbers, Alterations and Division Theorems}
\author{Karl Schwede}
\author{Kevin Tucker}
\begin{abstract}
Given an ideal $\ba \subseteq R$ in a (log) $\bQ$-Gorenstein $F$-finite ring of characteristic $p > 0$, we study and provide a new perspective on the test ideal $\tau(R, \ba^t)$ for a real number $t > 0$.  Generalizing a number of known results from the principal case, we show how to effectively compute the test ideal and also describe $\tau(R, \ba^t)$ using (regular) alterations with a formula analogous to that of multiplier ideals in characteristic zero.  We further prove that the $F$-jumping numbers of $\tau(R, \ba^t)$ as $t$ varies are rational and have no limit points, including the important case where $R$ is a formal power series ring. Additionally, we obtain a global division theorem for test ideals related to results of Ein and Lazarsfeld from characteristic zero, and also recover a new proof of Skoda's theorem for test ideals which directly mimics the proof for multiplier ideals.
\end{abstract}
\thanks{The first author was supported by NSF grant DMS \#1064485, NSF FRG grant DMS \#1265261 and NSF CAREER grant DMS \#1252860  and a Sloan Fellowship.}
\thanks{The second author was supported by NSF postdoctoral fellowship \#1004344 and NSF grant DMS \#1303077.}
\address{Department of Mathematics\\ The Pennsylvania State University\\ University Park, PA, 16802, USA}
\email{schwede@math.psu.edu}
\address{Department of Mathematics\\
University of Illinois at Chicago\\
322 Science and Engineering Offices (M/C 249)\\
851 S. Morgan Street\\
Chicago, IL 60607-7045, USA}
\email{kftucker@uic.edu}
\maketitle

\section{Introduction}

Suppose that $\ba$ is an ideal in a normal $\bQ$-Gorenstein domain $R$ essentially of finite type over a perfect field $k$.  When $k$ has characteristic zero, the multiplier ideals $\mathcal{J}(R, \ba^{t}) \subseteq R$ for real numbers $t \geq 0$ have been used to great effect (as in \cite[Chapter 9]{LazarsfeldPositivity2}). Similarly, when instead $k$ has characteristic $p > 0$, Hara and Yoshida introduced the test ideals $\tau(R, \ba^{t}) \subseteq R$ as positive characteristic analogs of multiplier ideals \cite{HaraYoshidaGeneralizationOfTightClosure} (\cf \cite{TakagiInterpretationOfMultiplierIdeals}).  In either case, these invariants measure both the singularities of $R$ and the subscheme defined by $\ba$ (relative to $t$). Moreover, Hara-Yoshida and Takagi showed that the multiplier ideal agrees with the test ideal after reduction to characteristic $p \gg 0$.  However, the connection between these invariants is far stronger than this result alone would suggest: test ideals and multiplier ideals often exhibit strikingly similar properties even in small characteristics.

This article develops a new perspective on test ideals of non-principal ideals with a number of applications.  For example, we further strengthen the association between multiplier and test ideals by generalizing the main results of \cite{BlickleSchwedeTuckerTestAlterations} and \cite{SchwedeTuckerZhang} from the case of a principal ideal to an arbitrary ideal.

\begin{theoremA*}[Test ideals via alterations, \autoref{thm.TestViaAlterations1}, \autoref{thm.TestViaAlterations2}]
\label{thmC.Alterations}
  Suppose that $R$ is a normal $\bQ$-Gorenstein domain essentially of finite type over a perfect field $k$, and $\ba$ is an ideal of $R$.  Then for all real numbers $t > 0$ and all sufficiently large regular alterations $\rho \colon W \to X = \Spec(R)$ with $\ba \cdot \O_{W} = \O_{W}(-H)$ locally principal (\textit{i.e.} those dominating a fixed alteration independent of $t$), we have
\[
\begin{array}{l@{=}l}
  \Tr_{\rho} \left( \rho_{*}\O_{W} (\lceil K_{W} - \rho^{*}K_{X} - tH \rceil ) \right) & \left\{
  \begin{array}{l}
    \mbox{the multiplier ideal $\mathcal{J}(R,\ba^{t})$ if $\Char(k) = 0$, and} \\
\mbox{the test ideal $\tau(R,\ba^{t})$ if $\Char(k)= p > 0$,}
  \end{array} \right.
\end{array}
\]
where $\Tr_{\rho} \colon \rho_{*} \omega_{W} \to \omega_{X}$ is the corresponding trace map.
\end{theoremA*}

While phrased in the above manner largely for comparison with characteristic zero, we remark that the new content of this result lies entirely in positive characteristic.  Furthermore, in this case, the above characterization of the test ideal extends naturally to all (log) $\bQ$-Gorenstein $F$-finite triples $(X, \Delta, \ba^{t})$:
there exists a regular alteration $\rho : W \to X$ with $\ba \cdot \O_W = \O_W(-H)$ such that
\begin{equation}
\label{eq.ConstantTauForSingleAlteration}
\Tr_{\rho} \big(\rho_* \O_W( \lceil K_W - \rho^*(K_X + \Delta) - tH \rceil)\big) = \tau(X, \Delta, \ba^t).
\end{equation}
for all $t \in \bR_{\geq 0}$, and it follows that the test ideal $\tau(X, \Delta, \ba^t)$ is the intersection of these images over all alterations $\rho : W \to X$,
\[
\tau(X, \Delta, \ba^t) = \bigcap_{\rho : W \to X} \Tr_{\rho} \big(\rho_* \O_W( \lceil K_W - \rho^*(K_X + \Delta) - tH \rceil)\big).
\]

Largely in contrast to the situation for multiplier ideals in characteristic zero, the test ideals of non-principal ideals have historically seemed far more complicated than test ideals of principal ideals.  So as to give an idea where the difficulties lie, recall that
the classical construction of the test ideal $\tau(R,\ba^{t})$ \cite{HaraYoshidaGeneralizationOfTightClosure,TakagiInterpretationOfMultiplierIdeals,BlickleMustataSmithDiscretenessAndRationalityOfFThresholds} is based upon a number of manipulations of ideals involving the Frobenius morphism or $p$th power map.
For example, the $p^e$th Frobenius or bracket power $\bb^{[p^e]}$ of an ideal $\bb \subseteq R$ is the expansion under the $e$-th iterate of Frobenius and is generated by the $p^e$th powers of elements of $\bb$.
In case $\bb$ is a principal ideal, the ordinary and Frobenius powers coincide $\bb^{[p^e]} = \bb^{p^e}$, simplifying the construction of $\tau(R,\ba^{t})$ when $\ba$ is principal.  However, in general the two powers $\bb^{[p^e]} \neq \bb^{p^e}$ are quite different, and subsequently the test ideal $\tau(R,\ba^{t})$ for non-principal $\ba$ is more mysterious.

In light of such difficulties, as with the results from \cite{BlickleSchwedeTuckerTestAlterations} and \cite{SchwedeTuckerZhang}  generalized above, a number of prominent results for test ideals have been known previously only in the principal case.  Our new perspective allows us to realize many of these statements in full generality, roughly showing that the test ideals of non-principal ideals behave nearly as well as those of principal ideals. As another example, consider that historically the test ideals of principal ideals have fared better with regards to $F$-jumping numbers.  Recall that a real number $t > 0$ is called an \emph{$F$-jumping number} if $\tau(R, \ba^{t - \varepsilon}) \neq \tau(R, \ba^t)$ for all $\varepsilon > 0$.  Based on the behavior of the similarly defined jumping numbers for multiplier ideals \cite{EinLazSmithVarJumpingCoeffs}, it was expected that $F$-jumping numbers are always rational and have no limit points (at least when $R$ is $\bQ$-Gorenstein, \cf \cite{UrbinatiDiscrepanciesOfNonQGorensteinVars}).  Indeed, this has been shown to be the case for arbitrary $\ba$ when $R$ is finite type over a field (for regular $R$ in \cite{BlickleMustataSmithDiscretenessAndRationalityOfFThresholds} and in general by \cite{BlickleSchwedeTakagiZhang,SchwedeTuckerZhang,BlickleTestIdealsViaAlgebras}). However, for more general rings, all previous proofs (again for regular $R$ in \cite{BlickleMustataSmithFThresholdsOfHypersurfaces,KatzmanLyubeznikZhangOnDiscretenessAndRationality} and in general by \cite{BlickleSchwedeTakagiZhang,SchwedeTuckerZhang}) have required $\ba$ to be principal.  This was particularly frustrating as the discreteness and rationality of $F$-jumping numbers for non-principal ideals in formal power series rings has remained elusive.  However, using our new characterization of the test ideal in the non-principal setting, we are able to prove discreteness and rationality in the $F$-finite (log) $\bQ$-Gorenstein case in full generality.

\begin{theoremB*}[Discreteness and rationality of $F$-jumping numbers, \autoref{thm.DiscAndRat}]
\label{thmB.discAndRat}
Suppose that $R$ is an $F$-finite normal domain, $\Delta$ is a $\bQ$-divisor on $X = \Spec(R)$ such that $K_X + \Delta$ is $\bQ$-Cartier, and $\ba \subseteq R$ is an ideal.  Then the set of $F$-jumping numbers of $\tau(X, \Delta, \ba^t)$ is a discrete set of rational numbers.
\end{theoremB*}
\noindent In particular, the $F$-jumping numbers of arbitrary ideals in formal power series rings are always discrete and rational.  This includes as a special case the rationality of the $F$-pure threshold, the positive characteristic analog of the log canonical threshold in characteristic zero.  It is interesting to note that, while Theorem A implies Theorem B a posteriori, in fact the discreteness and rationality of $F$-jumping numbers is a central ingredient in the proof of Theorem A.

The new perspective we are arguing for in this article can be summarized heuristically as follows. Before embarking upon Frobenius manipulations to compute the test ideal $\tau(R, \ba^t)$, one should first blowup the ideal $\ba$ (or alternately take a log resolution if available) so as to principalize $\ba$.  This in turn allows many of the methods from the principal case to go through at the technical expense of using relative cohomology vanishing theorems  -- in a manner reminiscent of their use for multiplier ideals in characteristic zero.
Roughly speaking, the proofs of Theorems A and B both proceed in this fashion.  However, they additionally require us to generalize effective tools for the computation of test ideals from the principal case to the non-principal case.


Recall from \cite{BlickleMustataSmithDiscretenessAndRationalityOfFThresholds} that if $R$ is regular, $\ba = \langle f \rangle$ is principal, and $t = a/p^e$, then $\tau(R, \ba^t) = (\ba^a)^{[1/p^e]}$ where $(\blank)^{[1/p^e]}$ indicates taking the image under the trace of the $e$-iterated Frobenius. The trace map can be described in this case via the Cartier isomorphism (see Section 2 for details).  The $(\blank)^{[1/p^e]}$-operation, alternately denoted $I_e(\blank)$ in \cite{KatzmanParameterTestIdealOfCMRings}, is highly computable and has been implemented in Macaulay2 by M.~Katzman (along with generalizations to not necessarily smooth ambient spaces using \cite{BlickleSchwedeTakagiZhang}).  However, the recipe $\tau(R, \ba^t) = (\ba^a)^{[1/p^e]}$ fails when $\ba$ is not principal, and the lack of a similarly effective description has been an obstruction for computing examples of $\tau(R, \ba^t)$ with non-principal $\ba$.  In order to prove Theorems A and B, we must first attempt to fill this gap in the following result.


\begin{theoremC*}[Effective computation of test ideals, \textnormal{\autoref{thm.ComputationOfTauSimple}}]
Suppose that $X$ is a Gorenstein\footnote{In fact, as stated, quasi-Gorenstein will suffice.} scheme, $\ba \subseteq \O_{X}$ is an ideal, and $t_0 \in \bQ_{>0}$ is a positive rational number. Set $\pi : Y \to X$ to be the normalized blowup of $Y$ with $\O_Y(-G) = \ba \cdot \O_Y$.  Let $\mathcal{N}$ denote the kernel of the natural map\footnote{Recall $\tau(\omega_Y, \Gamma)$ can be defined to be $\tau(Y, -K_Y + \Gamma)$; see Section~\ref{sec.TestVsParam} for further details.}
\[
F_* \big(\tau(\omega_Y, \pi^*K_{X}) \otimes \O_Y( (1-p)(\pi^*K_{X})) \big)
\to  \tau(\omega_Y, \pi^{*}K_{X} )
\]
induced by the trace map $F_* \omega_Y \to \omega_Y$.
Fix $e_1 \in \bZ_{>0}$ such that
\[
R^1 \pi_* (\mathcal{N}\tensor \O_Y(-f G)) = 0
\]
 for all $f \geq p^{e_1} t_0$ (possible since $-G$ is $\pi$-ample).
Then for any $t = a/p^b \in \bQ_{\geq t_{0}}$ and any $e \geq \max(e_1, b)$, we have
\[
\tau(X, \ba^t)= \tau(\omega_X, K_X, \ba^t) =  \Tr^e\Big( \pi_* F^e_* \big( \tau(\omega_Y, \pi^*K_{X}) \otimes \O_Y((1-p^e)\pi^*K_{X} - p^e tG) \big)\Big).
\]
\end{theoremC*}

\noindent
In practice, we expect that working values of $e_{1}$ can be detected via relative Castelnuovo-Mumford regularity \cite[Theorem 2]{OoishiCastelnuovoRegularityForGraded}.  Furthermore, we also obtain a more general result which allows for the effective computation of $\tau(X, \Delta, \ba^t)$ for all $t > 0$ and arbitrary triples $(X, \Delta, \ba)$ such that $K_X + \Delta$ is $\bQ$-Cartier in \autoref{thm.EffectiveTauComputationGeneral1}.  See \autoref{rem.Implementation} for some further discussion.

In addition to extending results from the principal case, our new perspective also leads to alternative proofs and entirely new statements for test ideals by allowing us to mimic previously unavailable arguments from characteristic zero.  Recall that, by combining a Koszul complex construction on a resolution together with the Kawamata-Viehweg vanishing theorem, one shows that multiplier ideals satisfy the Skoda-type relation $\mathcal{J}(R, \ba^{t}) = \ba \mathcal{J}(R,\ba^{t-1})$ for $t$ greater than the number of generators of $\ba$ (or even $\dim(R)$; see \cite[Section 9.6]{LazarsfeldPositivity2}).  Inspired by this result,  Hara and Takagi showed that the test ideal satisfies the analogous Skoda-type relation \cite[Theorem 4.1]{HaraTakagiOnAGeneralizationOfTestIdeals}. Hara and Takagi's relatively simple proof bears little resemblance to that from characteristic zero; more importantly, however, it is also comparatively weaker in the sense that it does not yield further global statements. Using our new perspective, in \autoref{prop.BasicSkoda} we give an alternative proof of Hara and Takagi's result which closely mimics the proof of Skoda's theorem for multiplier ideals.  Notably, we make use of the Koszul complex on the blowup but must replace the use of the Kawamata-Viehweg vanishing theorem in characteristic zero with asymptotic Fujita and Serre vanishing in positive characteristic.  Furthermore, this alternative line of argument then allows us to approach previously unknown statements for test ideals in the global setting.

 So as to state our result in detail, recall first that Ein and Lazarsfeld have used the Koszul complex construction for multiplier ideals mentioned above in the non-local setting to obtain the following global division theorem for sections of adjoint line bundles.


\begin{theorem*} \textnormal{(Global division theorem for multiplier ideals, \cite[Proposition 1.1(ii)]{EinLazEffectiveNullstellensatz} \cite[Theorem 9.6.31]{LazarsfeldPositivity2})}
Consider an ideal sheaf $\ba$ on a nonsingular $n$-dimensional projective variety $X$ over an algebraically closed field of characteristic zero.  Fix integral divisors $M$ and $L$ on $X$ such that $M - K_X$ is big and nef and $L$ has $r$ global sections
\[
s_{1}, \ldots, s_{r} \in H^{0}(X, \O_{X}(L) \tensor_{\O_{X}} \ba)
\]
generating $\O_{X}(L) \tensor_{\O_{X}} \bc$ for some reduction $\bc \subseteq \ba$ of $\ba$ (by replacing the $s_i$ with general linear combinations, we may assume that $r \leq n+1$).   Then for any $m \geq r$, any section
\[
s \in H^{0}(X, \O_{X}(M + mL) \tensor_{\O_{X}} \mJ(X, \ba^{m}))
\] can be expressed as a linear combination
\[
s = \sum h_{i}s_{i}
\]
with $h_{i} \in H^{0}(X, \O_{X}(M + (m-1)L) \tensor \mJ(X, \ba^{m-1}))$.
\end{theorem*}
\noindent In particular, the tensoring with the ideal $\mJ(X, \ba^m)$ can be viewed as a correction factor for which global sections can be pulled back via multiplication maps.  For example, if $\ba$ defines a smooth subvariety of codimension $d$, then $\mJ(X, \ba^m) = \ba^{m - d}$.  The proof Ein and Lazarsfeld's result heavily uses the Kawamata-Viehweg vanishing theorem \cite{KawamataVanishing,ViehwegVanishingTheorems}.

Another way to interpret the subspace $$H^{0}(X, \O_{X}(M + mL) \tensor_{\O_{X}} \mJ(X, \ba^{m})) \subseteq H^0(X, \O_X(M + mL))$$ is as follows.  If $\pi : Y \to X$ is a log resolution of $\ba$ with $\ba \cdot \O_Y = \O_Y(-G)$, then
\[
H^{0}(X, \O_{X}(M + mL) \tensor_{\O_{X}} \mJ(X, \ba^{m}))
\]
can simply be identified with the image of the (Grothendieck-)trace map
\[
H^0\big(Y, \O_Y(K_Y + m \pi^* L + \pi^* (M - K_X) - mG)\big) \xrightarrow{\Tr_{\pi}} H^0\big(X, \O_X(K_X + mL + M-K_X)\big).
\]

Motivated by this observation, our last main result is to obtain a version of Ein and Lazarsfeld's global division theorem in characteristic $p > 0$.  Firstly, we replace the multiplier ideal by the test ideal. Unfortunately, even with this replacement, the aforementioned vanishing theorems are false, and so we must correct not only for the \emph{local} singularities of $V(\ba)$ using the test ideal, but also for the potential \emph{global} failure of Kawamata-Viehweg vanishing.  We do this by further restricting our sections to take this into account, the new set of sections is denoted\footnote{In a previous version of this preprint, these sections were denoted by $P^0_+$.} by $P^0$.  First however, we state our theorem:


\begin{theoremD*} [\autoref{thm.MainResult}, \autoref{thm.BasicMainTheorem}]
Suppose that $X$ is a normal $n$-dimensional projective variety over an algebraically closed field of characteristic $p > 0$, and $\Delta \geq 0$ is a $\bQ$-divisor such that $K_X + \Delta$ is $\bQ$-Cartier.  Suppose that $\ba \subseteq \O_X$ is an ideal sheaf and $L$ is a Cartier divisor such that $\O_X(L) \tensor \bc$ is globally generated by sections $s_1, \dots, s_r \in \Gamma(X, \O_X(L) \tensor \bc)$ for some reduction $\bc \subseteq \ba$ of $\ba$ (by replacing the $s_i$ with general linear combinations, we may assume that $r \leq n+1$).  Fix $M$ a Cartier divisor on $X$ such that $M - K_X - \Delta$ is nef and big, and fix a positive integer $m \geq r$.  Then any section
\[
s \in P^0\Big( X, \O_X(M + m L) \tensor \tau(X, \Delta, \ba^m) \Big) \subseteq H^0(X, \O_X(M + mL ))
\]
can be expressed as a linear combination
\[
s = \sum h_i s_i
\]
with $h_i \in P^0\Big( X, \O_X(M + (m-1)L) \tensor \tau(X, \Delta, \ba^{m-1} )\Big)$.
\end{theoremD*}

In particular, instead of merely considering  $H^{0}(X, \O_{X}(M + mL) \tensor_{\O_{X}} \tau(X, \Delta, \ba^{m}))$, we instead consider a subspace
\[
P^{0}(X, \O_{X}(M+mL) \tensor_{\O_{X}} \tau(X, \Delta, \ba^{m})) \subseteq H^0(X, \O_X(M + mL)).
\]
This is the subspace obtained similarly to the multiplier ideal.  Indeed, suppose that $\pi : Y \to X$ is now the normalized blow-up of $\ba$ (or any further blow-up).  We define $P^{0}(X, \O_{X}(M+mL) \tensor_{\O_{X}} \tau(X, \ba^{m}))$ to be the sum of images of the (Grothendieck-)trace maps
\[
\begin{array}{rl}
& H^0\big(Y, \O_Y(K_Y + m p^e \pi^* L + \pi^* p^e (M - K_X - \Delta) - mp^e G - D)\big)\\
 \xrightarrow{\Tr_{F^e \circ \pi}} & H^0\big(X, \O_X(K_ X+ M - K_X + mL )\big).
 \end{array}
\]
for $e \gg 0$.
Here $F^e$ is simply the $e$-iterated Frobenius and $D$ is a sufficiently large effective divisor on $Y$.  While $D$ is large, when $e \gg 0$ the contribution of $-D$ is almost negligible compared to the other divisors (which are multiplied by $p^e$).

The first reasonable question one might ask is why there might be any sections of this form at all.  However, very recently there have a number of results proving that sections roughly of the form $P^0$ are abundant in characteristic $p > 0$.  For example, in \cite{SchwedeACanonicalLinearSystem} it was shown that under certain circumstances, these sections globally generate test ideal sheaves (precursors to this result were obtained in \cite{SmithFujitaFreenessForVeryAmple}, \cite{HaraACharacteristicPAnalogOfMultiplierIdealsAndApplications}, \cite{KeelerFujita}).  In \cite{SchwedeACanonicalLinearSystem}, it was also shown that these sections lift from subvarieties to ambient varieties via adjunction in some cases.  In \cite{MustataNonNefLocusPositiveChar}, it was shown that these sections can even be used to globally generate test ideals associated to linear series, analogous to \cite[Corollary 11.2.13]{LazarsfeldPositivity2}.  We briefly mention some other recent results showing that these special sections appear frequently.  In \cite{MustataSchwedeSeshadri}, it was shown that these sections can be detected by the Seshadri constant or an even finer positive characteristic analog of the Seshadri constant.  Additional applications of these sections can be found in \cite{CasciniHaconMustataSchwedeNumDimPseudoEffective,PatakfalviSemiPositivityInPosChar,HaconXuThreeDimensionalMMP}.  In the final section of this paper, we also explore the ubiquity of these sections in the case of curves, also see \cite{TanakaTraceMapExtendingSections}.  Theorem D can then be viewed as another piece in this puzzle, showing that these sections can be lifted via multiplication maps.

There is one key difference from our statement and Ein and Lazarsfeld's statement in \cite[Theorem 9.6.31]{LazarsfeldPositivity2} however.  There, if you have a section $s \in H^{0}(X, \O_{X}(M+mL))$ that vanishes to a sufficiently high degree along $\ba$, then $s$ is automatically contained in $H^{0}\big(X, \O_{X}(M+mL) \tensor_{\O_{X}} \mJ(X, \ba^{m})\big)$ (this follows from basic properties of multiplier ideals $\mJ(X, \ba^{m})$).  Unfortunately, we also have global arithmetic considerations as well and we do not see how to obtain the same result.  In particular, this restriction seems to prevent us from obtaining a global effective Nullstellensatz in characteristic $p > 0$, which was the main result of \cite{EinLazEffectiveNullstellensatz}.  Of course, in the affine setting, the effective Nullstellensatz is already known, even in characteristic $p > 0$ \cite{BrownawellBoundsForDegrees, KollarSharpEffectiveNullstellensatz}.

Again, we emphasize that one of the most interesting features of this last result is its proof.  In particular, our proof almost exactly mimics the global division theorem proof from characteristic $0$ in that we study a Skoda complex on a blowup $\pi : Y\to X$, push it down, and use vanishing statements to obtain the requisite surjectivity.  Similarly, in the future, we hope that the perspective and methods of this article will allow other characteristic zero arguments for multiplier ideals to yield new statements for test ideals in characteristic $p > 0$.


\vskip 9pt
\noindent{\it Acknowledgements: } The authors would like to thank Marc Chardin, Christopher Hacon Markus Lange-Hegermann, Mircea \mustata{} and Wenliang Zhang  for valuable discussions.  We thank the referee for numerous very helpful comments on a previous draft of this manuscript.  We also thank Wenliang Zhang for comments on a previous draft of this paper.  The particular construction of the special sections $P^0$ done by twisting by $-D$ in order to obtain asymptotic vanishing via Frobenius was first studied in this context in an unpublished work of C.~Hacon and the first author.

\section{Sketch of main ideas in simple cases}
\label{sec.SmoothVarieties}

In this section, we will give an overview of the main ideas of this paper for an ideal pair on a smooth ambient variety.  This simplified setting is both interesting in its own right, and renders transparent the essential ideas of the more general arguments.

We start by reviewing and generalizing a simple description of test ideals in this setting from \cite{BlickleMustataSmithDiscretenessAndRationalityOfFThresholds}.  We note that the definition of the test ideal we give here requires the ambient space to be smooth, see \autoref{sec.DescriptionOfTau} for the general case.  Suppose $X$ is a smooth $n$-dimensional variety over an algebraically closed field of characteristic $p > 0$.  Let $\omega_{X}$ denote the sheaf of $n$-forms on $X$, and denote by $F \: X \to X$ the (absolute) Frobenius morphism determined by taking regular functions to their $p$-th powers.

The key tool we will need is the \emph{trace} map $\Tr = \Tr_{X} \: F_{*} \omega _{X} \to \omega_{X}$.  This is a surjective map that can be described as a trace map for duality with respect to $F$, or equivalently as the map on $n$-forms via the Cartier isomorphism.  Given algebraic coordinates $x_{1}, \ldots, x_{n}$ on an open subset of $X$, the trace map is characterized by
\[
\Tr (F_{*} x_{1}^{i_{1}} \cdots x_{n}^{i_{n}} \, d x_{1} \wedge \cdots \wedge d x_{n}) =\left\{
\begin{array}{c}0 \mbox{ if } i_{j} \not\equiv -1 \mod (p) \mbox{ for some } j \medskip \\ x_{1}^{\frac{i_{1}-p+1}{p}} \cdots x_{n}^{\frac{i_{n}-p+1}{p}} \, d x_{1} \wedge \cdots \wedge d x_{n} \mbox{ otherwise } \end{array} \right. ,
\]
see \cite[Section 1.3]{BrionKumarFrobeniusSplitting}.  Iterating this map $e$-times (pushing forward via Frobenius as needed), we obtain a trace map $\Tr^{e} \: F^{e}_{*} \omega_{X} \to \omega_{X}$.

If $\ba \subseteq \O_{X}$ is an ideal sheaf one can use the surjectivity of trace to show that the images
\[
\Tr^{e} (F^{e}_{*}(\ba^{ p^{e} } \omega_{X} )) \subseteq \Tr^{e+1} (F^{e+1}_{*} (\ba^{p^{e+1} } \omega_{X}))
\]
are increasing subsheaves of $\omega_{X}$, and hence must stabilize for $e \gg 0$ by the Noetherian property.  As $\omega_{X}$ is invertible, there is an ideal $\tau(X, \ba)$ called the \emph{test ideal of $\ba$} such that
\[
\Tr^{e} (F^{e}_{*}(\ba^{p^{e} } \omega_{X} )) = \tau(X,\ba) \omega_{X}
\]
for all sufficiently large $e \gg 0$.  Alternately, the test ideal $\tau(X, \ba)$ can also be described as the smallest nonzero ideal $J \subseteq \O_X$ such that
\begin{equation}
\label{eq.EasySmallestTau}
\Tr^b(F^b_* (\ba^{(p^b- 1)} J)) \subseteq J
\end{equation}
for all $b \geq 0$. We now give a slight variation on the former description.

\begin{proposition}
\label{prop:alternatedescriptionsmoothspace}
Suppose that $\ba$ is an ideal sheaf on a smooth algebraic variety $X$, $\pi \: Y \to X$ any proper birational morphism such that $Y$ is normal and $\ba \O_{Y} = \O_Y(-G)$ is the locally principal ideal sheaf of an effective Cartier divisor $G$, and $D$ an effective Cartier divisor on $Y$ such that $D \geq K_{Y} - \pi^*(K_X) + (d+1)G$ where $d = \dim X$.  Then
\begin{equation}
\label{eq.EasyAscendingTestIdeal}
\Tr^{e}( F^{e}_{*} \pi_{*} \O_{Y}( K_{Y} - p^{e} G - D ) ) = \tau(X, \ba) \omega_{X}
\end{equation}
for all $e \gg 0$.
\end{proposition}
\begin{proof}
We may assume that $X$ is affine and that $K_X = 0$.  For simplicity, we set $\bb_e := \pi_{*} \O_{Y}( K_{Y} - p^{e} G - D )$ which is an actual ideal sheaf since $K_X = 0$ and $K_Y \geq 0$ is exceptional.  We first claim that the images from \autoref{eq.EasyAscendingTestIdeal} ascend as $e$ increases and so eventually stabilize.  It is sufficient to show that $\bb_{e} \subseteq \Tr(F_* \bb_{e+1})$ or in other words that $\bb_{e}^{[p]} \subseteq \bb_{e+1}$, \cf \cite[Lemma 1.6]{FedderFPureRat}.  Let $f \in \bb_{e}$, so that $\Div_Y(f) + K_Y - p^e G - D \geq 0$.  Thus $f^p \in \bb_e^{[p]}$ and $\Div_Y(f^p) + pK_Y - p^{e+1} G - pD \geq 0$.  On the other hand $pK_Y - p^{e+1}G - pD \leq K_Y - p^{e+1}G - D$ since $D \geq K_Y$ and the claim is proven.

We now show that $\bb_e \subseteq \ba^{p^e}$ which will gives us the containment $\subseteq$ based on the description of the test ideal above.  Now, certainly
\[
\bb_e = \pi_{*} \O_{Y}( K_{Y} - p^{e} G - D ) \subseteq \pi_* \O_Y(-p^e G - (d+1)G) = \overline{\ba^{p^e + d+1 }}
\]
from our choice of $D$.  But by the Brian{\c c}on-Skoda theorem \cite{LipmanTeissierPseudoRational,LipmanSathaye,HochsterHunekeTC1}, we know $\overline{\ba^{p^e + d+1 }} \subseteq \ba^{p^e}$.

For the reverse inclusion we use the characterization of the test ideal given in \autoref{eq.EasySmallestTau}.
Fix $e \gg 0$ and then notice that
\[
\begin{array}{rcl}
 \Tr^b\Big(F^b_* \big(\ba^{(p^b - 1)} \Tr^{e}( F^{e}_{*} \bb_e )\big)\Big) &
\subseteq & \Tr^{b+e}\Big(F^{b+e}_* \pi_* \O_Y(K_Y - (p^e + p^e(p^b - 1))G - D)\Big)\\
&\subseteq & \Tr^{b+e}\Big(F^{b+e}_* \pi_* \O_Y(K_Y - p^{e+b} G - D)\Big)\\
&= & \Tr^{b+e}\big(F^{b+e}_* \bb_{b+e}\big)\\
&= & \Tr^{e}\big(F^{e}_* \bb_{e} \big).
\end{array}
\]
\end{proof}

The advantage of this generalization is that it naturally allows one to make use of certain cohomology vanishing theorems.  In the notation of the Proposition, we are in fact free to take the divisor $D$ to be arbitrarily large; in particular, when $\pi$ is projective, we may take $-D$ with as much relative positivity as desired so as to apply vanishing theorems.  For instance, if $\pi$ is the normalized blowup, this can easily be achieved by using that $-G$ itself is relatively ample.  It is precisely these vanishings that lead to our main results.  Note also that we have chosen a $D$ above that makes the proof easy, and likely not the smallest $D$ one could choose.

In order to obtain effective computation of test ideals via the above characterization, it is essential to bound the $e \gg 0$ required for stabilization.  Roughly speaking, the next result does this where $\pi : Y \to X$ is the normalized blowup of $\ba$.  Along the way, we make use of the parameter test sheaf $\tau(\omega_Y)$ \cite{SmithTestIdeals,BlickleSchwedeTuckerTestAlterations} which replaces both $D$ and $K_Y$ in the above description.
Recall that since $Y$ is normal and $\omega_Y$ is reflexive, the map $\Tr : F_* \omega_U \to \omega_U$ on the smooth locus $U \subseteq Y$ extends to all of $Y$.  We then have that $\tau(\omega_Y)$ is by definition the smallest non-zero subsheaf $J \subseteq \omega_Y$ such that $\Tr(F_* J) \subseteq J$.  In fact, it thus follows that $\Tr(F_* \tau(\omega_Y)) = \tau(\omega_Y)$ since, if it was not surjective, the image would be a smaller ideal satisfying the same condition.  Similarly, for any effective $D$ on $Y$ and all $d \gg 0$, one has $\Tr^d(F_*^d \tau(\omega_Y)(-D)) = \tau(\omega_Y)$ (for example, see \cite[Proposition 2.2(4)]{SchwedeTuckerZhang}).
\begin{theorem}[Effective computation of test ideals]
\label{thm.EasyEffectiveComputation}
With notation as above, set $\mathcal{N}$ to be the kernel of $\Tr : F_* \tau(\omega_Y) \to \tau(\omega_Y)$.  Fix $e > 0$ such that
\begin{equation}
\label{eq.EasyVanishingEffectiveComputation}
R^1 \pi_* \big(\mathcal{N}\tensor \O_Y(-dG)\big) = 0
\end{equation}
for all $d \geq p^{e}$ (which is possible since $-G$ is $\pi$-ample).  Then
\[
\tau(X, \ba) \omega_X = \Tr^e\big(F^e_* \pi_*  (\tau(\omega_Y)\otimes\O_Y( - p^e G))\big).
\]
\end{theorem}
\begin{proof}
First, we may clearly assume $X$ is affine and that $K_X = 0$.
Fix $D$ sufficiently large and Cartier such that $\Tr^{e}( F^{e}_{*} \pi_{*} \O_{Y}( K_{Y} - p^{e} G - D ) ) = \tau(X, \ba)$ for $e \gg 0$.
By making $D$ bigger if necessary, we can assume $\O_Y(K_Y-D) \subseteq \tau(\omega_Y)$.  We also know that
\[
\Tr^{d}\big(F^d_* (\tau(\omega_Y) \otimes \O_Y(- D))\big) \subseteq \Tr^{d}\big(F^d_* ( \omega_Y \otimes \O_Y(- D))\big) \subseteq \tau(\omega_Y)
\]
by the definition of $\tau(\omega_Y)$ for any $d > 0$.

Let us first show how to get rid of $K_{Y}$ and $D$ at the expense of incorporating $\tau(\omega_{Y})$ in our description of the test ideal $\tau(X,\ba)$ -- but without worrying about bounding $e > 0$.
As mentioned immediately preceding the statement of the theorem, we may fix $d \gg 0$ such that
\[
F^d_* (\tau(\omega_Y) \otimes \O_Y(- D)) \xrightarrow{\Tr^d} \tau(\omega_Y)
\]
is surjective.  If $\mathcal{K}$ denotes the kernel of this map, then $R^1 \pi_* \big(\mathcal{K} \tensor \O_Y(-p^e G)\big) = 0$ for $e \gg 0$ by relative Serre vanishing, and so the composition
\[
F^d_*\pi_*  (\tau(\omega_Y) \otimes \O_Y(- D -p^{e+d} G)) \xrightarrow{\subseteq} F^d_* \pi_* (\omega_Y(- D - p^{e+d} G)) \xrightarrow{\Tr^d} \pi_* (\tau(\omega_Y)(-p^e G))
\]
is surjective.  Thus, applying $F^{e}_{*}(\blank)$ and $\Tr^{e}(\blank)$, we may conclude that
\[
\Tr^{e+d} \left( F^{e+d}_* \pi_* (\omega_Y(- D - p^{e+d} G))\right) = \Tr^{e}\left( F^{e}_{*}\pi_* (\tau(\omega_Y)(-p^e G))\right)
\]
 and
  it immediately follows that $\Tr^e \big( F^e_* \pi_* (\tau(\omega_Y)\otimes\O_Y(-p^e G) )\big) = \tau(X, \ba) \omega_X$ for $e \gg 0$ by \autoref{prop:alternatedescriptionsmoothspace}.

At this point, all that remains is to bound $e$. To that end, note that for the $e$ defined in the statement of the theorem, the vanishing \autoref{eq.EasyVanishingEffectiveComputation} implies that
\[
\begin{array}{rcl}
 \Tr^e \big( F^e_* \pi_* (\tau(\omega_Y)\otimes \O_Y(-p^e G)) \big) &
= & \Tr^{e+1} \big( F^{e+1}_* \pi_* (\tau(\omega_Y)\otimes \O_Y(-p^{e+1} G)) \big)\\
&= & \Tr^{e+2} \big( F^{e+2}_* \pi_* (\tau(\omega_Y)\otimes \O_Y(-p^{e+2} G)) \big) \\
&= & \ldots
\end{array}
\]
and so the theorem now follows.
\end{proof}

In the simple case under consideration in this section (for an ideal sheaf $\ba$ on a smooth variety $X$), Theorem A can readily be seen to follow from the above result.  In this situation (ignoring the scaling coefficient $t$), Theorem A asserts that there exists an alteration $\eta : W \to X$ with $\O_W(-H) = \ba \cdot \O_W$ such that $\tau(X, \ba) \omega_X = \Tr_{\eta}\big(\eta_* \O_W(K_W - H) \big)$.  This is direct from the above theorem if one picks $\gamma : W \to Y$ such that $\Tr_{\gamma}(\gamma_* \O_W(K_W)) = \tau(\omega_Y)$, whose existence is guaranteed by \cite{BlickleSchwedeTuckerTestAlterations}, and then sets $\eta = \pi \circ \gamma$.  Similarly, after incorporating the scaling coefficient $t$ into the above effective computation result in \autoref{thm.EffectiveTauComputationGeneral1}, we will be able to prove Theorem B on the discreteness and rationality of $F$-jumping numbers by using the above characterization to reduce to the principal case on the normalized blowup.


Let us now briefly sketch the main idea in Theorem B in our setting.  The technical core is \autoref{lem.NoAccumulationFromBelow} which shows that if $t = a/b \in \Q$ where $b$ is relatively prime to $p$, then $\tau(R, \ba^s)$ is constant for $s \in (t - \varepsilon, t)$ and $1 \gg \varepsilon > 0$ (compare with \cite{BlickleMustataSmithFThresholdsOfHypersurfaces,KatzmanLyubeznikZhangOnDiscretenessAndRationality}).  We already know the result in the principal case, which gives that $\tau(\omega_Y, (t - \varepsilon)G)$ is constant for $1 \gg \varepsilon > 0$ where $\pi : Y \to X$ is the normalized blowup and $\ba \cdot \O_Y = \O_Y(-G)$.  Using the effective computation of test ideals above and the formula
\begin{equation}
\label{eq.SimpleFormulaTrTau}
\Tr\big(F_* \tau(X, \ba^t) \omega_X\big) = \tau(X, \ba^{t/p}) \omega_X
\end{equation}
 it is easy to see that
\begin{equation}
\label{eq.EasyImageGeneralCoefficient}
\Tr^e F^e_* \pi_* \big( \tau(\omega_Y) \tensor \O_Y(-t(p^e - 1)G)\big) = \tau(X, \ba^{t( {p^e - 1 \over p^e})})\omega_X
\end{equation}
for all sufficiently divisible $e$ (\textit{i.e.} whenever $t(p^{e} - 1) \in \Z$).  Then setting $e = e_1 + e_2$ with both $e_1, e_2$ sufficiently divisible we consider the factorization of $\Tr^e$:
\[
\begin{array}{rl}
& F^{e_1 + e_2}_* \pi_* \big( \tau(\omega_Y) \tensor \O_Y(-t(p^{e_1+e_2} - 1)G)\big)\\
\xrightarrow{\beta} & F^{e_2}_* \pi_* \big( \tau(\omega_Y, t({p^{e_1} - 1 \over p^{e_1}})G) \tensor \O_Y(-t (p^{e_2} - 1) G)\big)\\
= & F^{e_2}_* \pi_* \big( \tau(\omega_Y, (t-\varepsilon)G) \tensor \O_Y(-t (p^{e_2} - 1) G)\\
\xrightarrow{\alpha} & \omega_X.
\end{array}
\]
The map $\beta$ comes from applying \autoref{eq.SimpleFormulaTrTau} on $Y$ and then twisting and applying $F^{e_2}_* \pi_*$.  Again, the equality after $\beta$ follows from the fact that we already know discreteness in the principal case.
The image of $\alpha \circ \beta$ is $\tau(X, \ba^{t({p^{e_1+e_2} - 1 \over p^{e_1+e_2}})})\omega_X$ by a  slight generalization of \autoref{eq.EasyImageGeneralCoefficient}.  The main trick is to show that $\beta$ is surjective (for any $e_1$ sufficiently large and divisible and chosen after $e_2$) using Serre vanishing.  Granting this, the image of $\alpha$ is also $\tau(X, \ba^{t({p^{e_1+e_2} - 1 \over p^{e_1+e_2}})})\omega_X$.  On the other hand, the image of $\alpha$ is independent of $e_1$ and so we can increase $e_1$ while leaving $e_2$ fixed and thus show that $\tau(X, \ba^{t - \varepsilon})$ is constant.

Next, let us sketch the alternative proof of Skoda's theorem for test ideals, namely $\tau(X,\ba^{m}) = \ba \tau(X,\ba^{m - 1})$ for $m$ at least the number of generators of $\ba$, as mentioned in the introduction. Suppose momentarily we are in the local setting where $X$ is affine and $\ba$ has a reduction generated by $s_{1}, \ldots, s_{r} \in \O_{X}$.  If once more $\pi : Y \to X$ is a proper birational morphism such that $Y$ is normal and $\ba \O_{Y} = \O_{Y}(-G)$, then  $\pi^{*} s_{1}, \ldots, \pi^{*}s_{r}$ are globally generating sections of $\O_{Y}(-G)$.  This implies $(\pi^{*}s_{1})^{p^{e}}, \ldots, (\pi^{*}s_{r})^{p^{e}}$ globally generate $\O_{Y}(-p^{e}G)$ for any $e > 0$, and we may form the corresponding Koszul complex
\begin{equation}
\label{eq.BabyKoszulComplex}
0 \to \sF_{r} \to \sF_{r-1} \to \cdots \sF_{1} \to \sF_{0} \to 0
\end{equation}
where $\sF_{i} = \O_{Y}(ip^{e}G)^{\oplus {r \choose i}}$ and each of the maps are essentially given (up to sign) as multiplication by the sections $(\pi^{*}s_{j})^{p^{e}}$. Note that this complex is exact (locally) since $\sF_0$ is invertible, \cite[Theorem 1.6.5]{BrunsHerzog}.  Since this complex is a locally free resolution of the (flat) sheaf $\sF_{0}$, this complex remains exact after tensoring by any quasicoherent sheaf on $Y$.

Let $m \geq r$ be an integer and $D \geq K_{Y} -\pi^{*}K_{X}+(d+1)mG$ an effective Cartier divisor on $Y$. Tensoring the Koszul complex constructed in \autoref{eq.BabyKoszulComplex} by the (not necessarily invertible) sheaf $\O_{Y}(  K_{Y} - p^{e}(\pi^{*}K_{X} + m G) - D )$ preserves exactness, and the $i$-th entry becomes $$\sG_{i} = \O_{Y}(  K_{Y} -p^{e}(\pi^{*}K_{X} + (m - i)G) - D )^{\oplus {r \choose i}}.$$  Using a relative version of Fujita's vanishing theorem \cite[Theorem 1.5]{KeelerAmpleFiltersOfInvertibleSheaves}, since the divisor $-p^{e}(\pi^{*}K_{X} + (m - i)G)$ is relatively nef for all $i \leq m$ and any $e$, we may choose $D$ sufficiently large and $\pi$-antiample so that $R^{j} \pi_{*} \sG_{i} = 0$ for all $i$, any $j > 0$, and arbitrary $e > 0$.  Hence our complex remains exact after applying $\pi_{*}( \blank)$ (start from the left and work right via short exact sequences), and as $F_{*}^{e}(\blank)$ is exact ($F$ is affine) we have that the complex
\[
0 \to F^{e}_{*}\pi_{*} \sG_{r} \to F^{e}_{*}\pi_{*} \sG_{r-1} \to \cdots \to F^{e}_{*}\pi_{*} \sG_{1} \to F^{e}_{*}\pi_{*} \sG_{0} \to 0
\]
is exact on $X$.  Furthermore, after having applied $F^{e}_{*}( \blank)$, we may view the arrows as given by multiplying by $s_{1}, \ldots, s_{r}$.  Taking images under $\Tr^{e}$ preserves exactness on the right, giving a surjection for $e \gg 0$ by Proposition \ref{prop:alternatedescriptionsmoothspace}
\[
\tau(X, \ba^{m - 1})^{\oplus r} \to[( \; s_{1} \; s_{2} \; \cdots \; s_{r} )] \tau(X, \ba^{m})
\]
whence we immediately recover the following well-known result.

\begin{proposition} \cite[Theorem 4.1]{HaraTakagiOnAGeneralizationOfTestIdeals}
\label{prop.BasicSkoda}
  Suppose $X$ is a nonsingular affine variety, and $s_{1}, \ldots, s_{r} \in \O_{X}$ generate a reduction of an ideal $\ba$.  Then
\[
\tau(X,\ba^{m}) = \sum_{i} s_{i} \tau(X,\ba^{m - 1})
\]
for all $m \geq r$.
\end{proposition}

Roughly the same idea can easily allow us to also handle log $\bQ$-Gorenstein triples $(X, \Delta, \ba^t)$.  Of course, the novelty lies not in the above statement, but rather in its proof -- which directly mimics the proof Skoda's theorem for multiplier ideals given in \cite[Section~9.6]{LazarsfeldPositivity2}.  Furthermore, by following the same line of argument, we shall soon arrive at a positive characteristic analog of the global division theorem shown therein.

Let us now move to the global setting and consider an ideal sheaf $\ba$ on a nonsingular projective variety $X$.  Fix integral divisors $A$ and $L$ on $X$ such that $A$ is ample\footnote{In the notation of the introduction, $M = K_X + A$.} and such that $L$ has $r$ global sections
\[
s_{1}, \ldots, s_{r} \in H^{0}(X, \O_{X}(L) \tensor_{\O_{X}} \ba)
\]
generating $\O_{X}(L) \tensor_{\O_{X}} \bc$ for some reduction $\bc$ of $\ba$.  Consider any proper birational map $\pi \: Y \to X$ from a normal variety $Y$ onto $X$ that dominates the blowup of $X$ along $\ba$, so that $\ba \O_{Y} = \O_{Y}(-G)$ for some effective Cartier divisor $G$ on $Y$.

Our main goal is to show that some special sections of $\sM := \O_{X}(K_{X} + mL + A)$  vanishing along $\tau(X,\ba^{m})$ for some integer $m \geq r$ can be written as a linear combination of the $s_{i}$ -- and the approach will mimic that from above.  We will essentially use the perturbation divisor $D$ in the description of the test ideal from Proposition~\ref{prop:alternatedescriptionsmoothspace} together with the positivity of $A$ to force certain vanishings to hold and arrive at an exact complex -- the image of which under trace will give the desired result.  However, this means we can only hope to show the division theorem for those sections that are themselves always in the image of the trace map.  We note that this type of idea has appeared in various places, such as \cite{SmithFujitaFreenessForVeryAmple} and \cite{HaraACharacteristicPAnalogOfMultiplierIdealsAndApplications}, where one forms the section ring with respect to a line bundle and then takes graded pieces of the test ideal on that section ring.  One advantage of our formulation is it makes it more convenient to work with multiple line bundles on several varieties simultaneously.

To that end, setting $\sL^{X, \ba^{m}}_{e, \pi, D} := \O_{Y}(K_{Y} - p^{e}(\pi^{*}K_{X} + mG) - D)$ for an effective Cartier divisor $D \geq K_{Y} - \pi^{*}K_{X} + (d+1)mG$, we have
from \autoref{prop:alternatedescriptionsmoothspace} that
\[
\sum_{e \gg 0} \Tr^{e} (\sM \tensor_{\O_{X}} F^e_{*} \pi_{*} \sL^{X,\ba^{m}}_{e,\pi,D}) = \sM \tensor_{\O_{X}} \tau(X,\ba^{m}).
\]
In particular, $\Tr^{e}$ also induces a map on global sections
\[
\Tr^{e}\left(H^{0}(Y, (\pi^{*}\sM)^{p^{e}} \tensor_{\O_{Y}} \sL^{X, \ba^{m}}_{e, \pi, D})\right) \subseteq H^{0}(X, \sM \tensor_{\O_{X}} \tau(X,\ba^{m}))
\]
and we can only hope to show the division theorem for those sections in
\[
P^{0}(X, \sM \tensor_{\O_{X}} \tau(X, \ba^{m}) ): = \bigcap_{D} \sum_{e \gg 0} \Tr^{e}\left(H^{0}(Y, (\pi^{*}\sM)^{p^{e}} \tensor_{\O_{Y}} \sL^{X, \ba^{m}}_{e, \pi, D})\right)
\]
where the intersection is over all possible effective Cartier divisors $D$ on $Y$.  Note that this intersection does stabilize, as it is an intersection inside the finite dimensional vector space $H^{0}(X, \sM \tensor_{\O_{X}} \tau(X,\ba^{m}))$.  We also notice that we need the $\sum_{e \gg 0}$ since a priori it is unclear whether the images form an ascending chain as $e$ varies.
We now come to the statement of the positive characteristic analog of the global division theorem.

\begin{theorem}
\label{thm.BasicMainTheorem}
Consider an ideal sheaf $\ba$ on a nonsingular projective variety $X$.  Fix integral divisors $A$ and $L$ on $X$ such that $A$ is ample and $L$ has $r$ global sections
\[
s_{1}, \ldots, s_{r} \in H^{0}(X, \O_{X}(L) \tensor_{\O_{X}} \ba)
\]
generating $\O_{X}(L) \tensor_{\O_{X}} \bc$ for some reduction $\bc$ of $\ba$.   Then for any $m \geq r$, any section
\[
s \in P^{0}(X, \O_{X}(K_{X}+mL + A) \tensor_{\O_{X}} \tau(X, \ba^{m}))
\]
can be expressed as a linear combination
\[
s = \sum h_{i}s_{i}
\]
with $h_{i} \in P^{0}(X, \O_{X}(K_{X} + (m-1)L + A) \tensor \tau(X, \ba^{m-1}) )$.
\end{theorem}

\begin{proof}
Consider any projective birational map $\pi \: Y \to X$ from a normal variety $Y$ that dominates the blowup of $X$ along $\ba$, so that $\ba \O_{Y} = \O_{Y}(-G)$ for some effective Cartier divisor $G$ on $Y$.
The sections $\pi^{*} s_{1}, \ldots, \pi^{*}s_{r}$ globally generate sections of $\O_{Y}(\pi^{*}L-G)$, hence also $(\pi^{*}s_{1})^{p^{e}}, \ldots, (\pi^{*}s_{r})^{p^{e}}$ globally generate $\O_{Y}(p^{e}(\pi^{*}L -G))$ for any $e > 0$. Begin by forming the corresponding Koszul complex
\[
0 \to \sF_{r} \to \sF_{r-1} \to \cdots \sF_{1} \to \sF_{0} \to 0
\]
where $\sF_{i} = \O_{Y}(-ip^{e}(\pi^{*}L-G))^{\oplus {r \choose i}}$ and each of the maps are essentially given (up to sign) as multiplication by the sections $(\pi^{*}s_{j})^{p^{e}}$. As before, this complex is a locally free resolution of the (flat) sheaf $\sF_{0} = \O_{Y}$, and so remains exact after tensoring by any quasicoherent sheaf on $Y$.  Set
\[
\Lambda_{j} = K_{Y} + jp^{e} (\pi^{*}L - G) + p^{e} \pi^{*} A - D
\]
for $j = 0, \ldots, m$. After we tensor the Koszul complex above by $\O_{Y}(\Lambda_{m})$, the $i$-th entry in the complex becomes $\sG_{i} = \O_{Y}(\Lambda_{m-i})^{\oplus {r \choose i}}$.

Since $\pi^{*}L - G$ is globally generated, it is certainly nef.  Furthermore, since $A$ is ample on $X$, we may take $e$ and $D$ sufficiently large (\textit{i.e.} $-D$ sufficiently $\pi$-ample) so that $p^{e}\pi^{*}A - D$ is ample.  Possibly increasing $e$ and $D$ further, we may once more apply the relative and global versions of Fujita's vanishing theorem \cite{KeelerAmpleFiltersOfInvertibleSheaves} to guarantee
\[
R^{j} \pi_{*} \O_{Y}(\Lambda_{m-i}) = 0
\qquad
H^{j}(Y, \O_{Y}(\Lambda_{m-i})) = 0
\]
for all $j > 0$ and all $i \geq 0$.  This gives that the complex
\[
0 \to F^{e}_{*}\pi_{*} \sG_{r} \to F^{e}_{*}\pi_{*} \sG_{r-1} \to \cdots \to F^{e}_{*}\pi_{*} \sG_{1} \to F^{e}_{*}\pi_{*} \sG_{0} \to 0
\]
is exact, and furthermore that it remains exact after taking global sections.  Thus, we have once more a surjective map
\begin{equation}
\label{eq:finalsurjection}
\Tr^{e}\left( H^{0}(X, F^{e}_{*}\pi_{*}\sG_{1})\right) \to[( \; s_{1} \; s_{2} \; \cdots \; s_{r} \; )] \Tr^{e}\left( H^{0}(X, F^{e}_{*}\pi_{*}\sG_{0})\right)
\end{equation}
where our notation on the left indicates that the trace map has been applied individually to each direct summand of
\[
H^{0}(X, F^{e}_{*}\pi_{*}\sG_{1}) = \left[ H^{0}(X,F^{e}_{*}\pi_{*} \O_{Y} (\Lambda_{m-1}))\right]^{\oplus r}.
\]
We then have that both
\[
\sum_{e \gg 0} \Tr^{e}\left( H^{0}(X, F^{e}_{*}\pi_{*}\sG_{1})\right) = \left[P^{0}(X, \O_{X}(K_{X} + (m-1)L + A) \tensor_{\O_{X}} \tau(X, \ba^{m-1}))\right]^{\oplus r}
\]
and
\[
\sum_{e \gg 0} \Tr^{e}\left( H^{0}(X, F^{e}_{*}\pi_{*}\sG_{0})\right) = P^{0}(X, \O_{X}(K_{X} + mL + A) \tensor_{\O_{X}} \tau(X, \ba^{m}))
\]
hold, and the desired conclusion now follows immediately from the surjectivity of \eqref{eq:finalsurjection}.
\end{proof}

At this point, our goal for the remainder of this article is essentially to push forward the above arguments to the general setting.  Once again, the first step is to generalize previously known descriptions of the test ideal so as to gain access to certain cohomology vanishing theorems.

\section{Alternate description of test ideals}
\label{sec.DescriptionOfTau}

In this section we introduce two alternate descriptions of test ideals which will motivate what we do later (and gives a posteriori motivation for the definitions in \autoref{sec.SmoothVarieties}).  The reader who is not interested in these formalities is invited to skip ahead to \autoref{prop.TauDescriptionViaBlowup} and take either as the \emph{definition} of the test ideal.  The second description will be useful in the proof of \autoref{thm.EffectiveTauComputationGeneral1}, while the first description will appear in \autoref{thm.MainResult}.

\begin{convention}
\label{conv.MainConvention}
Throughout this paper, all schemes $X$ will be Noetherian, separated and of equal characteristic $p > 0$.  They will additionally be assumed $F$-finite, meaning that the (absolute) Frobenius morphism $F : X \to X$ is finite.  This implies that all our schemes are locally excellent \cite{KunzOnNoetherianRingsOfCharP}.  We further assume that all schemes have dualizing complexes $\omega_{X}^{\mydot}$, which is automatic in the case of $F$-finite affine schemes by \cite{Gabber.tStruc}.  Finally, we assume throughout that $F^! \omega_X^{\mydot} \cong \omega_X^{\mydot}$, which is true on sufficiently small affine charts of any $F$-finite scheme, as well as on schemes of finite type over any scheme for which this property holds (\textit{e.g.} varieties over a perfect field); for additional discussion of this last condition, see \cite{BlickleSchwedeTakagiZhang}.
\end{convention}

\begin{definition}[The trace map]
Suppose that $X$ is an integral scheme satisfying the above conditions.  Then Grothendieck dual to the Frobenius map $\O_X \to F^e_* \O_X$ we obtain a map $F^e_* \omega_X^{\mydot} \to \omega_X^{\mydot}$ on the dualizing complex (shifted so that the first non-zero cohomology is in degree $-\dim X$).  By taking the $-\dim X$ cohomology, we obtain a map of canonical modules $F^e_* \omega_X \to \omega_X$ which we denote by $\Tr^e$ and call the \emph{trace map}.  More generally, for any morphism $\rho : Y \to X$ between integral schemes of the same dimension, one can construct $\Tr_{\rho} : \rho_* \omega_Y \to \omega_X$ in a similar manner; see \cite[Proposition 2.18]{BlickleSchwedeTuckerTestAlterations} for further discussion.
\end{definition}

\begin{definition}
\label{triple.definition}
A \emph{(log $\bQ$-Gorenstein) triple} $(X, \Delta, \ba^t)$ is a normal integral scheme
$X$ together with an effective $\bQ$-divisor $\Delta$ on $X$ such that $K_X + \Delta$ is $\bQ$-Cartier and an ideal sheaf $\ba$ on $X$ with a non-negative real coefficient $t$.
\end{definition}

\begin{definition}
\label{def.TauDefinition}
Suppose that $(X, \Delta, \ba^t)$ is a log $\bQ$-Gorenstein triple with $X = \Spec R$ affine.  Then the \emph{test ideal of $(X, \Delta, \ba^t)$} is denoted by $\tau(X, \Delta, \ba^t)$ and defined to be the unique smallest non-zero ideal $J$ satisfying the following condition:  for every $e > 0$ and every section $\phi \in \Hom_{\O_X}(F^e_* \O_X(\lceil (p^e - 1)\Delta \rceil), \O_X)$, one has
    \[
        \phi(F^e_* (\ba^{\lceil t(p^e - 1) \rceil} J)) \subseteq J.
    \]
    Here we view $F^e_* (\ba^{\lceil t(p^e - 1) \rceil} J) \subseteq F^e_* \O_X \subseteq F^e_* \O_X(\lceil (p^e - 1) \Delta \rceil)$.  It is straightforward to verify that the test ideal is compatible with localization, and hence this definition can be extended to triples where $X$ is not affine by gluing in the obvious manner.
\end{definition}

Let us also briefly record the following lemma for later use.

\begin{lemma}
\label{lem.UniformImages}
Let $\eta : Y \to W$ be a proper birational map between normal varieties. Suppose that $\Gamma$ is a $\bQ$-Cartier divisor on $W$, and $E \geq 0$ is a Weil divisor on $Y$.  For each $e > 0$, there is a natural inclusion of sheaves
\[
\eta_* \O_Y(K_Y - \lfloor p^e \eta^* \Gamma \rfloor - E) \subseteq \O_W(K_W - \lfloor p^e \Gamma \rfloor).
\]
Furthermore, there exists a divisor $D> 0$ on $W$ such that
\[
\eta_* \O_Y(K_Y - \lfloor p^e \eta^* \Gamma \rfloor - E) \supseteq \O_W(K_W - \lfloor p^e \Gamma \rfloor - D)
\]
for all $e>0$.
\end{lemma}
\begin{proof}
Both statements immediately reduce to the case that is $W$ is affine, which we now assume.
After identifying the function fields $ K(Y) = K(W)$, we have a natural inclusion $\eta_{*}\O_{Y}(B) \subseteq \O_{Y}(\eta_{*}B)$ for any Weil divisor $B$ on $Y$.  Using that $E$ is effective, this immediately gives the first statement.  For the second, start by taking any effective Cartier divisor $C$ on $W$ such that $\eta^{*}C \geq E - K_{Y}$.  Now let $D$ be any effective Cartier divisor on $W$ such that $D \geq K_{W} + C + \Supp \Gamma$.  We have
\[
K_{W} - \lfloor p^{e}\Gamma \rfloor - D \leq - \lfloor p^{e} \Gamma \rfloor - \Supp \Gamma - C \leq - p^{e} \Gamma - C
\]
and also
\[
-p^{e} \eta^{*}\Gamma - \eta^{*}C \leq \lceil -p^{e} \eta^{*}\Gamma \rceil +K_{Y} - E = K_{Y} - \lfloor p^{e}\eta^{*}\Gamma \rfloor - E.
\]
But now $f \in \O_W(K_{W} - \lfloor p^{e}\Gamma \rfloor - D)$ if and only if $\Div_W(f) + K_{W} - \lfloor p^{e}\Gamma \rfloor - D \geq 0$ which implies that $\Div_W(f) - p^{e} \Gamma - C \geq 0$.  Therefore, with such an $f$, $\Div_Y(\eta^* f) -p^{e} \eta^{*}\Gamma - \eta^{*}C \geq 0$, which implies that $\Div_Y(\eta^* f) + K_{Y} - \lfloor p^{e}\eta^{*}\Gamma \rfloor - E \geq 0$. This completes the proof of the second statement.
\end{proof}

We now give an alternate description of the test ideal.  In particular, this description can be interpreted as a description of the test ideal, similar to that of a multiplier ideal, except that instead of a resolution, we take any blowup and then repeatedly apply Frobenius.

\begin{proposition}
\label{prop.TauDescriptionViaBlowup}
Let $(X, \Delta, \ba^{t})$ be a log-$\Q$-Gorenstein triple. 
Consider any proper birational map $\pi \: Y \to X$ from a normal variety $Y$ that dominates the blowup of $X$ along $\ba$, so that $\ba \O_{Y} = \O_{Y}(-G)$ for some effective Cartier divisor $G$ on $Y$.
Then we have
\[
\tau(X, \Delta, \ba^t) = \bigcap_{D} \bigcap_{e_0 \geq 0} \left( \sum_{e \geq e_0} \Tr^{e}\left( F^e_* \pi_*\sL^{X,\Delta,\ba^{t}}_{e, \pi, D}   \right)\right)
\]
where the intersection ranges over all effective divisors $D$ on $Y$ and we set
\[
\sL^{X,\Delta,\ba^{t}}_{e, \pi, D} = \O_{Y}( \lceil K_{Y} - p^e(\pi^{*}(K_{X}+\Delta) + t G)\rceil - D).
\]
Furthermore, this intersection stabilizes for all sufficiently large divisors $D$, and for some $e_0$ depending on $D$.  In other words, there exists an effective divisor $D$ and some $e_0 > 0$ such that
\[
\tau(X, \Delta, \ba^t) =  \sum_{e \geq e_0} \Tr^{e}\left( F^e_* \pi_*\sL^{X,\Delta,\ba^{t}}_{e, \pi, D}   \right)
\]
and moreover the same equality holds after increasing the size of $D$ or $e_{0}$ (although one is free to take $e_0 = 0$ for $D$ large enough as well).
\end{proposition}

\begin{proof}
We can assume that $X$ is affine.  It is well known that there exists a Cartier divisor $B_0 > 0$  on $X$ (the vanishing locus of a test element) and some $e_0 \geq 0$ (or any $e_0 \geq 0$) such that for all Cartier $B \geq B_0$ we have
\begin{equation}
\label{eq.TauClassicalSum}
\begin{array}{rcl}
 \tau(X, \Delta, \ba^t) &
= & \sum_{e \geq e_0} \Tr^e F^e_* \left( \ba^{\lceil t p^e \rceil} \cdot \O_X(\lceil K_X - p^e (K_X + \Delta) - B \rceil)\right) \\
&= & \sum_{e \geq e_0} \Tr^e F^e_* \left( \overline{\ba^{\lfloor t p^e \rfloor}} \cdot \O_X(\lceil K_X - p^e (K_X + \Delta) - B \rceil)\right)
\end{array}
\end{equation}
\noindent
For example, see \cite[Lemma 2.1]{HaraTakagiOnAGeneralizationOfTestIdeals} or \cite[Definition-Proposition 3.3]{BlickleSchwedeTakagiZhang}.  Note the rounding is slightly different from other sources, but these differences may all easily be absorbed into $B_0$.  Similarly, the fact that we are taking the integral closure of $\ba^{\lfloor tp^e \rfloor}$ can be absorbed into $B_0$ as well using the tight-closure {B}rian\c con-Skoda theorem \cite[Theorem 5.4]{HochsterHunekeFRegularityTestElementsBaseChange}.

Fix $H > 0$ on $X$ such that
\[
\pi_* \O_Y(\lceil K_Y  - p^e\pi^*( K_X + \Delta) \rceil) \supseteq \O_X(\lceil K_X - p^e(K_X + \Delta) - H \rceil)
\]
for all $e \geq 0$ by \autoref{lem.UniformImages}.  Thus observe that
\[
\begin{array}{rl}
& \pi_* \O_Y(\lceil K_Y  - p^e(\pi^*( K_X + \Delta) + t G) - \pi^* B \rceil) \\
\supseteq & \pi_* \O_Y(\lceil K_Y  - p^e\pi^*( K_X + \Delta) - \pi^* B \rceil - \lceil p^e t G \rceil ) \\
\supseteq & \ba^{\lceil p^e t \rceil} \O_X(\lceil K_X - p^e(K_X + \Delta) - H - B\rceil).
\end{array}
\]
Applying $\Tr^{e} F^e_*$ to both sides and summing up, we then arrive at the containment
\[
\tau(X, \Delta, \ba^t) \subseteq \bigcap_{D, e_0 \geq 0} \left( \sum_{e \geq e_0} \Tr^{e}\left( F^e_* \pi_*\sL^{X,\Delta,\ba^{t}}_{e, \pi, D}   \right)\right)
\]
by noting that, given any $D$, we can always find $B \geq B_0$ such that $\pi^* B \geq D$. Also note that we may assume that $e_0 \gg 0$.

The reverse containment is similar; however, note that it suffices to show for some fixed $D$ that
\begin{equation}
\label{eq:stabilizationcontainment}
\sum_{e \geq 0} \Tr^{e}\left( F^e_* \pi_*\sL^{X,\Delta,\ba^{t}}_{e, \pi, D}   \right) \subseteq \tau(X, \Delta, \ba^t)
\end{equation}
which will simultaneously verify (all of) the stabilization statements as well.
To that end, we need show a claim which plays the same role as \autoref{lem.UniformImages} did above.
\begin{claim}
\label{claim.TauDescriptionViaBlowup}
There exists a divisor $D'$ on $Y$ such that, for all $e \geq 0$,
\[
\pi_* \O_Y(\lceil K_Y - p^e(\pi^*(K_X + \Delta) + t G) - D'\rceil) \subseteq \overline{\ba^{\lfloor p^e t \rfloor}} \cdot \O_X(\lceil K_X - p^e(K_X + \Delta) \rceil).
\]
\end{claim}

\noindent
Before proving the claim, let us finish the proof of Proposition \ref{prop.TauDescriptionViaBlowup}.  By taking $D = \pi^* B_{0} + D'$, twisting the containment from the claim by $\O_{X}(-B_{0})$, applying $F^{e}_{*}(\blank)$ and $\Tr^e(\blank)$ to both sides, and summing up over $e \geq 0$, we see that \eqref{eq:stabilizationcontainment} follows immediately from \eqref{eq.TauClassicalSum}.
\end{proof}

\begin{proof}[Proof of Claim]
  The statement is local, so we continue to assume that $X$ is affine, and may further assume $K_{X} \geq 0$.  Let $n$ be the index of $K_{X} + \Delta$, and let $C$ be an effective Cartier divisor on $X$ such that $C \geq n(K_{X} + \Delta) - K_X$.  Then for all $e > 0$ we have
\[
-\lfloor \frac{p^{e}}{n} \rfloor n (K_{X}+ \Delta) - C \leq K_X - p^{e}(K_{X}+ \Delta) \leq \lceil K_{X} - p^{e}(K_{X} + \Delta) \rceil
\]
and so it suffices to find $D'$ satisfying
\[
\pi_* \O_Y(\lceil K_Y - p^e(\pi^*(K_X + \Delta) + t G) - D'\rceil) \subseteq \overline{\ba^{\lfloor p^e t \rfloor}} \cdot \O_X( -\lfloor \frac{p^{e}}{n} \rfloor n (K_{X}+ \Delta) - C)
\]
for all $e \geq 0$.  Now, since $\overline{\ba^{\lfloor p^e t \rfloor}} = \pi_{*}\O_{Y}(-\lfloor p^{e} t \rfloor G)$ and $-\lfloor \frac{p^{e}}{n} \rfloor n (K_{X}+ \Delta) - C$ is Cartier, we have by the projection formula that
\[
\begin{array}{rcl}
\overline{\ba^{\lfloor p^e t \rfloor}} \cdot \O_X( -\lfloor \frac{p^{e}}{n} \rfloor n (K_{X}+ \Delta) - C) &=& \overline{\ba^{\lfloor p^e t \rfloor}} \otimes \O_X( -\lfloor \frac{p^{e}}{n} \rfloor n (K_{X}+ \Delta) - C) \\
&= & \pi_{*}\O_{Y}(-\lfloor \frac{p^{e}}{n} \rfloor n \pi^{*}(K_{X}+ \Delta) - \pi^{*}C -\lfloor p^{e} t \rfloor G)
\end{array}
\]
so it suffices to find $D'$ satisfying
\[
\lceil K_Y - p^e(\pi^*(K_X + \Delta) + t G) - D'\rceil \leq -\lfloor \frac{p^{e}}{n} \rfloor n \pi^{*}(K_{X}+ \Delta) - \pi^{*}C -\lfloor p^{e} t \rfloor G
\]
for all $e \geq 0$.  Then notice that
\[
\begin{array}{rl}
& \lceil K_Y - p^e(\pi^*(K_X + \Delta) + t G) \rceil + \lfloor \frac{p^{e}}{n} \rfloor n \pi^{*}(K_{X}+ \Delta) + \pi^{*}C + \lfloor p^{e} t \rfloor G\\
\leq & K_Y - \lfloor p^e \pi^* (K_X + \Delta)\rfloor - \lfloor p^e t G \rfloor + \lfloor \frac{p^{e}}{n} \rfloor n \pi^{*}(K_{X}+ \Delta) + \pi^{*}C + \lfloor p^{e} t \rfloor G\\
\leq & K_Y + \pi^* C - \lfloor ({p^e \over n} - \lfloor {p^e \over n} \rfloor )n\pi^* (K_X + \Delta) \rfloor - \lfloor (p^e t - \lfloor p^e t \rfloor) G\rfloor\\
\leq & K_Y + \pi^* C
\end{array}
\]
and so any $D' \geq K_Y + \pi^* C$ verifies the claim.
\end{proof}

\begin{remark}
The only novel part about \autoref{prop.TauDescriptionViaBlowup} is the fact that we use the divisor $D$ on $Y$ as a replacement for a test element on $X$.  The reason that this is useful is that, if $\pi$ is projective, we may choose $-D$ to be relatively ample.  It follows then that the divisor $- p^{e} (\pi^*(K_X + \Delta) + t G) - D $ is also relatively ample (since $-\pi^* (K_X + \Delta)$ and $-G$ are both relatively nef).
\end{remark}

\begin{remark}[Sufficiently divisible $e > 0$ and $(p^e - 1)$]
\label{rem.TauDivisibleE}
One can vary the setup of \autoref{prop.TauDescriptionViaBlowup} in a number of different ways.  For example, in the notation found therein, for each $e$ and $D$ set
\[
\sM^{X,\Delta,\ba^{t}}_{e, \pi, D} = \O_{Y}( \lceil K_{Y} - (p^e - 1)(\pi^{*}(K_{X}+\Delta) + t G)\rceil - D).
\]
It is easy to see that $\sL^{X,\Delta,\ba^{t}}_{e, \pi, D}$ may be replaced with $\sM^{X,\Delta,\ba^{t}}_{e, \pi, D}$ throughout by absorbing the difference into $D$.
Furthermore, additionally suppose now that the index of $K_X + \Delta$ is not divisible by $p$.
One may then also easily show that
\begin{equation}
\label{eq.TauByDivisibleEAndMinus1}
\tau(X, \Delta, \ba^t) = \bigcap_{D} \bigcap_{e_0 \geq 0} \left( \sum_{\substack{e = l e_0}} \Tr^{e}\left( F^e_* \pi_*\sM^{X,\Delta,\ba^{t}}_{e, \pi, D}   \right)\right).
\end{equation}
The point here is that this description concerns only sufficiently divisible values of $e$.  To see that this is possible, simply observe in this case that \autoref{eq.TauClassicalSum} holds for a sum over sufficiently divisible $e$ (for example, by \cite[Proposition 3.9]{BlickleSchwedeTakagiZhang}) and repeat the arguments of the \autoref{prop.TauDescriptionViaBlowup} for such divisible $e$.

One advantage of \eqref{eq.TauByDivisibleEAndMinus1} is that since $p$ does not divide the index of $K_X + \Delta$, and if furthermore $t$ is rational and $p$ does not appear in its denominator, one can arrange that $(1-p^e)(K_X + \Delta + tG)$ is integral for all sufficiently divisible $e$ and so the roundings in the definition of $\sM^{X,\Delta,\ba^{t}}_{e, \pi, D}$ are unnecessary.
\end{remark}


Next we observe that Keeler's relative version of Fujita's vanishing theorem, \cite{KeelerAmpleFiltersOfInvertibleSheaves,FujitaVanishingTheoremsForSemiPositive} gives us the following relative vanishing theorem for test ideals.

\begin{proposition}[Relative vanishing for test ideals]
\label{prop.RelativeVanishingForTau}
In the notation of Proposition \ref{prop.TauDescriptionViaBlowup}, assume further that $\pi : Y \to X$ is projective. Then there exists a divisor $D \geq 0$ on $Y$ and an $e_0 \geq 0$ such that
\[
\tau(X, \Delta, \ba^t) = \sum_{e \geq e_0} \Tr^e\Big( F^e_* \pi_* \big(\O_Y(\lceil K_Y - p^e \pi^*(K_X + \Delta + tG) - D \rceil) \big)  \Big)
\]
and also such that $R^i \pi_* \big(\O_Y(\lceil K_Y - p^e \pi^*(K_X + \Delta + tG) - D \rceil)\big) = 0$ for all $i > 0$ and all $e \geq e_0$.
\end{proposition}

\begin{proof}
The statement is local, so we may assume that $X$ is affine.  Fix $D'$ to be any effective $\pi$-antiample divisor.  Suppose that $n(K_X + \Delta)$ is Cartier for some integer $n$.  Next consider the set of divisors on $Y$
\[
\mathcal{Q} = \big\{ \lfloor i(\pi^*(K_X + \Delta) + tG) \rfloor - \lfloor i/n \rfloor n \pi^*(K_X + \Delta)  - \lfloor t i \rfloor G \;\big|\; i \in \bN \big\}.
\]
Note that the divisors in this set are integral, with bounded coefficients, and are supported on a finite set of prime divisors.  Hence $\mathcal{Q}$ is a finite set of integral divisors, $\mathcal{Q} = \{ F_1, \ldots, F_d \}$.

Set
\[
\sF = \bigoplus_{i = 1}^d \O_Y( K_Y - F_i ) = \bigoplus_{i = 1}^d \sF_i.
\]
By relative Fujtita vanishing \cite[Theorem 1.5]{KeelerAmpleFiltersOfInvertibleSheaves},
there exists an $m_0 > 0$ such that
\[
R^i \pi_* \left( \sF \tensor \O_Y(-mD') \tensor \sN \right)= 0
\]
for any $m \geq m_0$, $i > 0$, and $\pi$-nef invertible sheaf $\sN$.   Thus the same vanishing also holds for the summands $\sF_i$.  By choosing a sufficiently large effective Cartier divisor $B$ on $X$, we may assume $D = mD' + \pi^* B$ satisfies the stabilization conditions in Proposition \ref{prop.TauDescriptionViaBlowup} for some $e_0 \geq 0$.  This gives the first equality, and we need only prove the vanishing statements as well.

But note that, for any $e \geq 0$, we have
\[
\begin{array}{rl}
& \O_Y(\lceil K_Y - p^e (\pi^*(K_X + \Delta) + tG) -  D \rceil) \\
= & \O_Y(K_Y - \lfloor p^e (\pi^*(K_X + \Delta) + tG) \rfloor - D)\\
= & \O_Y(K_Y - F_j - \lfloor p^e /n \rfloor n \pi^*(K_X - \Delta) - \lfloor t p^e \rfloor G )
\end{array}
\]
where $F_j$ is the element of $\mathcal{Q}$ corresponding to $i = p^e$.  But now the higher cohomologies vanish as $G$ is relatively antiample and $\pi^*(K_X - \Delta)$ is relatively antinef.
\end{proof}

\begin{remark}
In a previous draft of this paper, we included another variant on the definition of the test ideal where the coefficient of $D$ was $p^e \varepsilon$ for some $1 \gg \varepsilon > 0$.  We have since been successful at rendering this complication unnecessary.
\end{remark}

\section{Test ideals vs parameter test modules}
\label{sec.TestVsParam}

This short section shows that test ideals of pairs are equal to parameter test modules, in the sense of \cite{SmithTestIdeals}, of slightly different pairs.  This is useful because the parameter test module behaves very naturally in certain change-of-variety operations (including alterations).  Hence we will use the parameter test modules in the sections that follow wherever we believe that it conceptually simplifies certain arguments.  This perspective was used previously in \cite{BlickleSchwedeTuckerTestAlterations} and \cite{SchwedeTuckerZhang}.

We begin with definitions.

\begin{definition}[Parameter test modules for rings and effective divisors]
\label{def.ParameterTestSubmodules}
Suppose that $X = \Spec R$ is a normal integral $F$-finite affine scheme, $\Gamma$ is an effective $\bQ$-divisor, $\ba \subseteq R$ is an ideal and $t \geq 0$ is a rational number.

Then we define the \emph{parameter test module}, denoted $\tau(\omega_X, \Gamma, \ba^{t})$, to be the unique smallest nonzero submodule of $J \subseteq \omega_X$ such that
for every $e \geq 0$ and every map
\[
\phi \in \Hom_{\O_X}\big(F^e_* (\omega_X( \lceil (p^e - 1) \Gamma \rceil)), \omega_X\big) \subseteq \Hom_{\O_X}(F^e_* \omega_X, \omega_X) = \langle \Tr^e\rangle_{ F^e_* \O_X} \cong F^e_* \O_X
\]
 we have that
\[
\phi\big(F^e_* ( \ba^{\lceil t(p^e - 1) \rceil}\cdot J)\big) \subseteq J.
\]
If $\Gamma = 0$ or $\ba = R$, we omit those terms from the notation and write $\tau(\omega_X,\ba^{t})$ or $\tau(\omega_X, \Gamma)$.
\end{definition}

Let us now prove that this submodule exists by proving that it coincides with the test ideal of an appropriate pair.

\begin{lemma}
\label{lem.ParamTestModAreTestIdeals}
With notation as in \autoref{def.ParameterTestSubmodules}, still in the affine case, choose $K_X$ such that $- K_X$ is effective and fix $\omega_X = \O_X(K_X)$ to be the corresponding submodule of $K(X)$, then
\[
\tau(\omega_X, \Gamma, \ba^{t}) = \tau(X, \Gamma-K_X, \ba^{t})
\]
as a submodule of $K(X)$ or equivalently setting $\Theta = \Gamma -K_X$,
\[
\tau(\omega_X, \Theta + K_X, \ba^t) = \tau(X, \Theta, \ba^t).
\]
In particular, the parameter test module exists.
\end{lemma}
\begin{proof}
Since $-K_X$ is effective and integral, we see that $J = \tau(X, \Gamma-K_X, \ba^{t})$ is a submodule of $\omega_X = \O_X(K_X) \subseteq \O_X$ (it is easy to see that the containment holds in codimension 1 which is sufficient since $\omega_X$ is reflexive).  On the other hand, observe that
\[
\begin{array}{rl}
& \Hom_{\O_X}\big(F^e_* (\omega_X( \lceil (p^e - 1) \Gamma \rceil)), \omega_X\big) \\
\cong & \Hom_{\O_X}\big(F^e_* (\O_X(K_X + \lceil (p^e - 1) \Gamma \rceil)) \tensor_{\O_X} \O_X(-K_X), \O_X\big) \\
\cong & \Hom_{\O_X}\big(F^e_* (\O_X( \lceil (p^e - 1) (\Gamma-K_X) \rceil)), \O_X\big).
\end{array}
\]
It follows that the maps which determine the parameter test module in \autoref{def.ParameterTestSubmodules} are exactly the maps used to define the test ideal in \autoref{def.TauDefinition}.  But now both the test ideal and test module are contained inside $\omega_X \subseteq \O_X$ and both are defined to be the smallest such submodule satisfying the same property.  Hence they are equal and the lemma is proven.
\end{proof}

Since $\tau(\omega_X, \Gamma, \ba^{t}) = \tau(X, \Gamma-K_X, \ba^{t})$, we can conclude numerous properties about the parameter test module automatically.  For example, the formation of the parameter test module automatically commutes with localization.    Likewise
\begin{equation}
\label{eq.TauOmegaUnchangedBySmallEpsilon}
\tau(\omega_X, \Gamma + \varepsilon D, \ba^t) = \tau(\omega_X, \Gamma, \ba^t)
\end{equation}
for all $1 \gg \varepsilon > 0$ and effective Cartier divisors $D > 0$, see \cite[Lemma 3.23]{BlickleSchwedeTakagiZhang}.

Another such property is the fact that if $D$ is any Cartier divisor on $X$, then
\begin{equation}
\label{eq.TauOmegaPrincipalSkoda}
\begin{array}{rl}
& \tau(\omega_X, D + \Gamma, \ba^t) \\
= & \tau(X, D+\Gamma-K_X, \ba^t)\\
 =&  \tau(X, \Gamma-K_X, \ba^t) \tensor \O_X(-D) \\
 = & \tau(\omega_X, \Gamma, \ba^t) \tensor \O_X(-D).
 \end{array}
\end{equation}
From this it follows that we can define the parameter test module for non-effective $\Gamma$ as follows.

\begin{definition} [Parameter test modules in general]
\label{def.ParameterTestSubmodulesNonEffective}
Working in the setting of \autoref{def.ParameterTestSubmodules} but without assuming that $\Gamma$ is effective, choose $D$ a Cartier divisor (working locally if necessary) such that $\Gamma+D$ is effective.  Then we define the parameter test module by the following formula:
\[
\tau(\omega_X, \Gamma, \ba^t) = \tau(\omega_X, \Gamma + D, \ba^t) \otimes \O_X(D).
\]
Note that the parameter test module is a fractional ideal once we fix $\omega_X \subseteq K(X)$.
It is easy to see that this is independent of the choice of $D$.
\end{definition}

Now in the non-affine setting, since the formation of the parameter test module commutes with localization, we can define the parameter test module of an \emph{arbitrary normal integral $F$-finite scheme} in the obvious way.

As mentioned at the start of the section, the transformation rules for parameter test modules are particularly transparent.  We fix the following notation, if $f : Y \to X$ is an alteration (for example, it could be a finite map) then $\Tr_f : f_* \omega_Y \to \omega_X$ is the Grothendieck trace, see \cite[Proposition 2.18]{BlickleSchwedeTuckerTestAlterations}.

\begin{lemma}[Transformation rules for parameter test modules]
\label{lem.TransformationRulesForParameterTestModules}
Suppose that $X$ is a integral normal $F$-finite scheme, $\Gamma$ is a $\bQ$-divisor on $X$, $\ba$ is an ideal sheaf and $t \geq 0$ is a real number.  Then
\begin{itemize}
\item[(a)]  if $f : Y \to X$ is a finite map where $Y$ is also integral and normal then
\[
\Tr_f\big(f_* \tau(\omega_Y, f^* \Gamma, (\ba \cdot \O_Y)^t)\big) = \tau(\omega_X, \Gamma, \ba^t).
\]
\item[(b)]  in the case of \textnormal{(a)} if $f$ is the $e$-iterated Frobenius $F^e : X \to X$, then we have
\[
\Tr^e\big(F^e_* \tau(\omega_X, \Gamma, \ba^t \big) = \tau(\omega_X, {1/p^e}\Gamma, \ba^{t/p^e}).
\]
\item[(c)]  if $\Gamma$ is $\bQ$-Cartier, then there exists a finite map, or if desired when $X$ is essentially of finite type over a field, a regular alteration  $f : Y \to X$ such that
\[
\Tr_f\big(f_* \omega_Y(\lceil  -f^* \Gamma \rceil)\big) = \tau(\omega_X, \Gamma).
\]
\end{itemize}
\end{lemma}
\begin{proof}
Part (a) is simply the main result of \cite{SchwedeTuckerTestIdealFiniteMaps} translated into parameter test modules.  We briefly explain how this is done using the notation of that paper. The idea is as follows, the ramification divisor, or more generally what is denoted there by $R_\Tt$, is built into $\omega_X$, $\omega_Y$ and the Grothendieck trace.  Indeed if one chooses any embeddings $\omega_X \subseteq K(X)$ and $\omega_Y \subseteq K(Y)$, then localizing the Grothendieck trace at the generic point induces a map $\Tt : K(Y) \to K(X)$.  This is the map to be used in \cite[Main Theorem (General Case)]{SchwedeTuckerTestIdealFiniteMaps}.

Part (b) is merely a special case of (a) once one notices the following.  The ideal $\ba \cdot F^e_* \O_X = F^e_* \ba^{[p^e]}$ has the same integral closure as $F^e_* \ba^{p^e}$ and so either yield the same test ideal, and thus parameter test module by \autoref{lem.ParamTestModAreTestIdeals}, also see \cite[Proposition 4.4]{BlickleSchwedeTuckerTestAlterations}.

Part (c) is the main result of \cite{BlickleSchwedeTuckerTestAlterations} stated in terms of parameter test modules.
\end{proof}

In this paper, we will exclusively deal with the case that $\Gamma$ is a $\bQ$-Cartier divisor.  Frequently, we will also assume that the index of $\Gamma$ is not divisible by $p$.  In such a case we have the following which we highlight because it will appear several times.

\begin{corollary}
\label{cor.SurjectionOfTraceOnTauOmega}
With notation as above, suppose that $\Gamma$ is a $\bQ$-divisor such that $(p^e - 1)\Gamma$ is Cartier for some integer $e > 0$.  Then
\[
\Tr^e\big(F^e_* \tau(\omega_X, \Gamma) \otimes \O_X( (1-p^e)\Gamma) \big) = \tau(\omega_X, \Gamma).
\]
\end{corollary}
\begin{proof}
Simply observe that
\[
\begin{array}{rl}
& \Tr^e\big(F^e_*( \tau(\omega_X, \Gamma) \otimes \O_X( (1-p^e)\Gamma) )\big) \\
= & \Tr^e\big(F^e_* \tau(\omega_X, \Gamma - (1-p^e)\Gamma)  \big)\\
= & \Tr^e\big(F^e_* \tau(\omega_X, p^e \Gamma ) \big)\\
= & \tau(\omega_X, \Gamma)
\end{array}
\]
where the final equality comes from \autoref{lem.TransformationRulesForParameterTestModules}(b).
\end{proof}

\section{Effective computation of test ideals}

Let $(X, \Delta, \ba^{t})$ be a log-$\Q$-Gorenstein triple.
 In this section of the paper, we describe an algorithm for computing $\tau(X, \Delta, \ba^t)$ that could in principal be implemented in a computer.  Working exclusively on the normalized blowup $\pi : Y \to X$ of $\ba$, roughly speaking the key point is to incorporate the parameter test sheaf \cite{SmithTestIdeals} so as to remove both $D$ and $\sum_{e \geq e_{0}}$ from the description of the test ideal in \autoref{prop.TauDescriptionViaBlowup}.
%
%
We utilize the parameter test module heavily in this section, see \autoref{sec.TestVsParam} for discussion.

\begin{theorem}[Effective test ideal computation]
\label{thm.EffectiveTauComputationGeneral1}
Let $(X, \Delta, \ba^t)$ be a log-$\Q$-Gorenstein triple, and write $t = a/p^b (p^c-1)$ for some integers $a,b,c > 0$ where additionally we have that $p^b(p^c-1)(K_X + \Delta)$ is Cartier.   Set $\pi : Y \to X$ to be the normalized blowup of $\ba$ with $\O_Y(-G) = \ba \cdot \O_Y$.  Let $\mathcal{N}$ denote the kernel of the natural map
\[
\begin{array}{rrl}
& F^c_* \big(\tau(\omega_Y, p^b\pi^*(K_X + \Delta) + p^b t G) \otimes \O_Y( (1-p^c)p^b(\pi^*(K_X + \Delta) + tG)) \big) & \\
& \cong F^c_* \big(\tau(\omega_Y, p^c p^b\pi^*(K_X + \Delta) + p^c p^b t G) & \\
& \twoheadrightarrow\tau(\omega_Y, p^b\pi^*(K_X + \Delta) + p^b tG) & \\
\end{array}
\]
induced by the trace map $F^c_* \omega_Y \to \omega_Y$, the surjectivity follows from \autoref{cor.SurjectionOfTraceOnTauOmega}.
Fix $e_1 = mc > 0$ a positive multiple of $c$ such that
\begin{equation}
\label{eq.VanishingEffectiveTauComputationGeneral1}
R^1 \pi_* (\mathcal{N} \tensor \O_Y(-(p^e - 1)p^b tG)) = 0
\end{equation}
for all $e = e_1 + nc$ where $n$ runs over all integers $\geq 0$ (which possible since $-G$ is $\pi$-ample).
Then for all such $e$
\[
\begin{array}{rl}
  & \tau(X, \Delta, \ba^t) \\
= &  \tau(\omega_X, K_X + \Delta, \ba^t) \\
= & \Tr^{e+b}\Big( \pi_* F^{e+b}_* \big( \tau(\omega_Y, p^b\pi^*(K_X + \Delta) + p^btG) \otimes \O_Y(p^b(1-p^e)(\pi^*(K_X + \Delta) + tG)) \big)\Big).
\end{array}
\]
\end{theorem}

\begin{proof}  The question is local so we may assume that $X$ is affine.
We next observe that we may assume that $b = 0$.  Indeed, by \autoref{lem.TransformationRulesForParameterTestModules}(b)
we know that for any $\bQ$-Cartier $\bQ$-divisor $\Gamma$, any $b \geq 0$  and rational $w > 0$
\[
\begin{array}{rcl}
 \Tr^b\Big(F^b_* \tau(\omega_X, \Gamma, \bb^w)\Big) 
& = & \tau(\omega_X, {1 \over p^b} \Gamma, \bb^{w/p^b}).
\end{array}
\]
Setting $\Gamma_X = K_X + \Delta$ and working with the parameter test module $\tau(\omega_X, \Gamma_X, \ba^t)$,    we may then replace $\Gamma_X$ by $p^b \Gamma_X$ and $t$ by $p^b t$ and from this point forward assume that $b = 0$.

For compactness of notation, set $\Gamma_Y := \pi^*(K_X + \Delta) + tG$ and observe that it is $\pi$-antiample.  By \autoref{prop.TauDescriptionViaBlowup} and \autoref{rem.TauDivisibleE} we now choose an effective Cartier divisor $B$ on $Y$ such that
\begin{equation}
\label{eq.tauEqualsSumMinusF}
\begin{array}{rcl}
 \tau(X, \Delta, \ba^t) &
 = & {\displaystyle{\sum_{\substack{e \gg 0 \\ \text{divisible}\\ \text{by $c$}}} \Tr^e\Bigg(\pi_* F^e_* \Big(\tau\big(\omega_Y, \Gamma_Y\big) \otimes \O_Y\big( (1-p^e)\Gamma_Y  - B\big)\Big)\Bigg)}.}
 \end{array}
\end{equation}
Note that the term $\tau(\omega_Y, \Gamma_Y)$, instead of $K_Y$, is harmless since it is constant and so can be absorbed into the divisor $B$.

Fix a $d \gg 0$, divisible by $c$, such that by \autoref{lem.TransformationRulesForParameterTestModules}(b) and \autoref{eq.TauOmegaUnchangedBySmallEpsilon}
\[
F^d_* \Big(\tau\big(\omega_Y, \Gamma_Y\big) \otimes \O_Y\big( (1-p^d)\Gamma_Y - B\big)\Big) \xrightarrow{\Tr^d}\!\!\!\!\!\!\to \tau(\omega_Y, \Gamma_Y + {1 \over p^d}B) = \tau(\omega_Y, \Gamma_Y).
\]
Note the $\Tr^d$ above is the trace on $Y$, not on $X$.
Let $\mathcal{K}$ denote the kernel of this map.  Let $e_2$ (a multiple of $c$) be such that
\[
R^1 \pi_* \big(\mathcal{K}\tensor \O_Y( (1-p^e)\Gamma_Y)\big) = 0
\]
for all $e = lc + e_2$, $l \geq 0$.  Therefore, for those same $e = lc + e_2$, we have that
\[
\pi_* F^{e+d}_* \Big(\tau\big(\omega_Y, \Gamma_Y\big) \otimes \O_Y\big( (1-p^{e+d})\Gamma_Y - B\big)\Big) \xrightarrow{F^e_* \Tr^d} \pi_* F^e_* \Big(\tau(\omega_Y, \Gamma_Y) \otimes \O_Y( (1-p^e)\Gamma_Y) \Big)
\]
is surjective.  Thus, for all such $e = lc + e_2$, by \autoref{cor.SurjectionOfTraceOnTauOmega} and the projection formula,
\[
\begin{array}{rl}
& \Tr^{e+d}\Bigg(\pi_* F^{e+d}_* \Big(\tau\big(\omega_Y, \Gamma_Y\big) \otimes \O_Y\big( (1-p^{e+d})\Gamma_Y - B\big)\Big) \Bigg)\\
= & \Tr^e\Bigg(\pi_* F^e_* \Big(\tau(\omega_Y, \Gamma_Y) \otimes \O_Y( (1-p^e)\Gamma_Y) \Big) \Bigg).
\end{array}
\]
Applying this to \autoref{eq.tauEqualsSumMinusF}, it follows that
\begin{equation}
\label{eq.tauEqualToSumForNow}
\begin{array}{rcl}
 \tau(X, \Delta, \ba^t)
 &=&  {\displaystyle{\sum_{\substack{e \gg 0 \\  \text{divisible}\\ \text{by $c$}}} \Tr^e\Bigg(\pi_* F^e_* \Big(\tau\big(\omega_Y, \Gamma_Y\big) \otimes \O_Y\big( (1-p^e)\Gamma_Y\big)\Big)\Bigg)}}
 \end{array}
\end{equation}
the difference from \autoref{eq.tauEqualsSumMinusF} being that $-B$ no longer appears.  Notice that by \autoref{cor.SurjectionOfTraceOnTauOmega} and the projection formula, we have maps induced by applying $\pi_*$ to $\Tr^c_Y$:
\[
\begin{array}{rl}
& \Tr^c\Bigg(\pi_* F^{e+c}_* \Big(\tau\big(\omega_Y, \Gamma_Y\big) \otimes \O_Y\big( (1-p^{e+c})\Gamma_Y\big)\Big)\Bigg)\\
\to & \pi_* F^{e}_* \Big(\tau\big(\omega_Y, \Gamma_Y\big) \Big)\otimes \O_Y\big( (1-p^{e})\Gamma_Y\big)\Big)
\end{array}
\]
and so the sum in \autoref{eq.tauEqualToSumForNow} is actually a sum of descending terms.  Hence
\[
\tau(X, \Delta, \ba^t) = \Tr^e\Bigg(\pi_* F^e_* \Big(\tau\big(\omega_Y, \Gamma_Y\big) \otimes \O_Y\big( (1-p^e)\Gamma_Y\big)\Big)\Bigg)
\]
for any $e \gg 0$ divisible by $c$.  Our next goal is to replace the $e \gg 0$ with an effective choice of $e > 0$.  We will use the vanishing \autoref{eq.VanishingEffectiveTauComputationGeneral1} to accomplish this.

Notice that for each $n > 0$, the map induced by $\Tr^c$ coming from \autoref{cor.SurjectionOfTraceOnTauOmega} and the projection formula
\begin{equation}
\label{eq.PushforwardSurjectivityUltimate1}
\begin{array}{rl}
& \pi_* F^{e_1+nc}_* \Big(\tau\big(\omega_Y, \Gamma_Y\big) \otimes \O_Y\big( (1-p^{e_1+nc})\Gamma_Y\big)\Big) \\
\to & \pi_* F^{e_1+(n-1)c}_* \Big(\tau\big(\omega_Y, \Gamma_Y\big) \otimes \O_Y\big( (1-p^{e_1+(n-1)c})\Gamma_Y\big)\Big)
\end{array}
\end{equation}
has cokernel equal to
\[
R^1 \pi_* F^{e_1+(n-1)c}_* \Big( \mathcal{N}\tensor \cO_Y\big( (1-p^{e_1+(n-1)c}) \Gamma_Y\big)\Big) = 0
\]
where the vanishing is \autoref{eq.VanishingEffectiveTauComputationGeneral1}.  Thus \autoref{eq.PushforwardSurjectivityUltimate1} is surjective.
Hence by repeatedly applying this argument, we have that
\[
\begin{array}{rl}
& \Tr^{e_1 + nc} \Bigg(\pi_* F^{e_1+nc}_* \Big(\tau\big(\omega_Y, \Gamma_Y\big) \otimes \O_Y\big( (1-p^{e_1+nc})\Gamma_Y\big)\Big)\Bigg)\\
= & \Tr^{e_1} \Bigg(\pi_* F^{e_1}_* \Big(\tau\big(\omega_Y, \Gamma_Y\big) \otimes \O_Y\big( (1-p^{e_1})\Gamma_Y\big)\Big)\Bigg),
\end{array}
\]
and so the theorem follows.
\end{proof}

\begin{remark}
\label{rem.Implementation}
The authors hope that this method might eventually allow the computation of test ideals $\tau(X, \Delta, \ba^t)$ to be implemented in a computer algebra system such as Macaulay2 \cite{M2}.  The main obstruction is the computation of the integer $e_1$ such that \mbox{$R^1 \pi_* (\mathcal{N} \tensor \O_Y(-(p^e - 1)p^b tG)) = 0$} for all $e = e_1 + nc$.  However, as suggested by Markus Lange-Hegermann on \begin{center}{\tt http://mathoverflow.net/questions/105333/effective-serre-vanishing},\end{center} one should be able to use relative\footnote{That is over a base ring instead of a field.} Castelnuovo-Mumford regularity \cite[Theorem 2]{OoishiCastelnuovoRegularityForGraded} to detect when a large enough $e_1$ has been obtained.  Alternately, the images computing $\tau(X, \Delta, \ba^t)$ are descending as $e$ increases, and so perhaps this can be played off against ascending chains of ideals which also compute the test ideal.  This may be easier to implement (although perhaps slower).  
\end{remark}

\begin{remark}
Theorem~\ref{thm.EffectiveTauComputationGeneral1} is \emph{not} true if instead one replaces the normalized blowup with a resolution of singularities.  Indeed, consider the D4 singularity $R = \mathbb{F}_2\llbracket x, y \rrbracket/\langle z^2+xyz+xy^2+x^2y\rangle$ given in \cite[Example 7.12]{SchwedeTuckerTestIdealFiniteMaps} (and originally coming from \cite{ArtinWildlyRamifiedZ2Actions}), with $\ba = R$ and $t = 1$.  Then if $\pi : Y \to X = \Spec R$ is a log resolution, we have $\tau(\omega_Y) = \omega_Y$ and so $\pi_* \tau(\omega_Y) = \omega_{X} = \O_X$.  But this example is $F$-pure, so $\Tr^e (F^e_* \O_X) = \O_X$ for all $e$.  However, the test ideal $\tau(\O_X)$ is the maximal ideal $\langle x, y \rangle$.  The reason our argument does not apply is because $\ba \cdot \O_Y = \O_Y$ is not $\pi$-ample.
\end{remark}

We also state a variant of the above which has additional similarities with the computations of test ideals as presented in \cite[Definition 2.9]{BlickleMustataSmithDiscretenessAndRationalityOfFThresholds}.  The main difference between this result and the previous one is here we pick $e$ largely independent of $t$ when $t = a/p^b$.

\begin{theorem}
\label{thm.ComputationOfTauSimple}
Suppose that $(X, \Delta, \ba)$ is a triple, $t_0 > 0$ is a positive rational number, and that $(p^c - 1) (K_X + \Delta)$ is Cartier for some $c > 0$.  Set $\pi : Y \to X$ to be the normalized blowup of $\ba$ with $\O_Y(-G) = \ba  \O_Y$.  Let $\mathcal{N}$ denote the kernel of the natural map
\[
\begin{array}{rl}
F^c_* \big(\tau(\omega_Y, \pi^*(K_X + \Delta)) \otimes \O_Y( (1-p^c)(\pi^*(K_X + \Delta))) \big) & \\
\to  \tau(\omega_Y, \pi^{*}(K_X + \Delta) ) & \\
\end{array}
\]
induced by the trace map $F^c_* \omega_Y \to \omega_Y$.
Fix $e_1 > 0$ a positive multiple of $c$ such that
\begin{equation}
\label{eq.VanishingComputationOfTauSimple}
R^1 \pi_* (\mathcal{N} \tensor \O_Y(-f G)) = 0
\end{equation}
 for all $f \geq p^{e_1} t_0$ (this is possible since $-G$ is $\pi$-ample).
Then for any $t = a/p^b \geq t_0$ and any $e$ a multiple of $c$ satisfying $e = kc \geq \max(e_1, b)$
\[
\begin{array}{rl}
  & \tau(X, \Delta, \ba^t) \\
= &  \tau(\omega_X, K_X + \Delta, \ba^t) \\
= &  \Tr^e\Big( \pi_* F^e_* \big( \tau(\omega_Y, \pi^*(K_X + \Delta)) \otimes \O_Y((1-p^e)\pi^*(K_X + \Delta) - p^e tG) \big)\Big).
\end{array}
\]
\end{theorem}
\begin{proof}
The strategy is essentially the same as in \autoref{thm.EffectiveTauComputationGeneral1}.
As before, assume $X$ is affine and we observe that by the argument of \autoref{rem.TauDivisibleE}, for some choice of Cartier divisor $B > 0$ on $Y$, if we set $\Theta = \pi^*(K_X + \Delta)$, we have:
\begin{equation}
\label{eq.tauEqualsSumMinusPtoED}
\begin{array}{rl}
& \tau(X, \Delta, \ba^t) \\
 = & {\displaystyle{\sum_{\substack{e \gg 0 \\ \text{divisible}\\ \text{by $c$}}} \Tr^e\Bigg(\pi_* F^e_* \Big(\tau\big(\omega_Y, \Theta   \big) \otimes \O_Y\big( (1-p^e)\Theta - tp^e G - B\big)\Big)\Bigg)}.}
 \end{array}
\end{equation}
Note any difference between $\tau(\omega_Y, \Theta)$ and $\O_Y(K_Y)$ can be absorbed into $B$ as before.
Likewise, as in the previous proof choose $d > 0$, divisible by $c$, such that by \autoref{lem.TransformationRulesForParameterTestModules}(b) and \autoref{eq.TauOmegaUnchangedBySmallEpsilon}
\[
F^d_* \Big( \tau(\omega_Y, \Theta) \otimes \O_Y( (1-p^d)\Theta - B) \Big) \xrightarrow{\Tr^d}\!\!\!\!\!\!\to \tau(\omega_Y, \Theta + {1 \over p^d} B) = \tau(\omega_Y, \Theta)
\]
surjects.  Let $\mathcal{K}$ denote the kernel of this map and let $e_2$, a multiple of $c$, be such that
\[
R^1 \pi_* \big(\mathcal{K} \tensor \O_Y( (1-p^e) \Theta - tp^e G)\big) = 0
\]
for all $e = lc + e_2 > b$.  Therefore, for those same $e$, by the projection formula we have that
\[
\begin{array}{rl}
& \pi_* F^{e+d}_* \Big(\tau\big(\omega_Y, \Theta \big) \otimes \O_Y\big( (1-p^{e+d})\Theta - tp^{e+d} G - B\big)\Big) \\
\xrightarrow{F^e_* \Tr^d} & \pi_* F^e_* \Big(\tau(\omega_Y, \Theta) \otimes \O_Y( (1-p^e)\Theta) - tp^e G \Big)
\end{array}
\]
is surjective.  Therefore
\[
\begin{array}{rl}
 \tau(X, \Delta, \ba^t)
 =  {\displaystyle{\sum_{\substack{e \gg 0 \\  \text{divisible}\\ \text{by $c$}}} \Tr^e\Bigg(\pi_* F^e_* \Big(\tau\big(\omega_Y, \Theta\big) \otimes \O_Y\big( (1-p^e)\Theta - p^e tG\big)\Big)\Bigg)}}.
 \end{array}
\]
Let $m$ be a multiple of $c$ with $m \geq \max( e_1, b)$. Note that for $e =kc \geq m$, the map
{
\medmuskip=0mu
\thickmuskip=0mu
\thinmuskip=0mu
\[
\pi_* F^e_* \Big(\tau{\hskip-0.75pt}\big(\omega_Y, \Theta{\hskip-1.5pt}\big) \otimes {\O_Y}{\hskip-1.5pt}\big( (1-p^e)\Theta - p^e tG\big)\Big) \xrightarrow{\Tr^{e-m}} \pi_* F^m_* \Big(\tau{\hskip-0.75pt}\big(\omega_Y, \Theta{\hskip-1.5pt}\big) \otimes \O_Y{\hskip-1.5pt}\big( (1-p^m)\Theta - p^m tG\big)\Big)
\]
}
surjects by \autoref{eq.VanishingComputationOfTauSimple}.
But these maps factor $\Tr^e$ and the claimed result follows immediately.
\end{proof}

\begin{remark}
\label{rem.ComputationOfTauSimpleDivP}
If the index of $K_X + \Delta$ is divisible by $p$, it is not difficult to account for that by using the formula $\Tr^e(\tau(\omega_X, K_X + \Delta, \ba^t)) = \tau(\omega_X, {1/p^e}(K_X + \Delta), \ba^{t/p^e})$ from \autoref{lem.TransformationRulesForParameterTestModules}(b).  We leave it to the reader to formulate this generalization.
\end{remark}

\section{Discreteness of $F$-jumping numbers}

In this section, we fix $X$ a normal $F$-finite scheme with ideal sheaf $\ba$.  We study $F$-jumping numbers of $\tau(X, \ba^t)$.  For basic setup and definition of terms, see \cite{BlickleSchwedeTakagiZhang} and \cite{SchwedeTuckerZhang}.  First, we recall the following Lemma.

\begin{lemma}\textnormal{(\cite[Lemma 3.23]{BlickleSchwedeTakagiZhang})}
\label{lem.EasyEpsilonJump}
Suppose that $X$ is normal, $\Delta \geq 0$ on $X$ is such that $K_X + \Delta$ is $\bQ$-Cartier and that $\ba$ ideal sheaf on $X$.  Then for any $t \geq 0$, there exists an $\varepsilon > 0$ such that
\[
\tau(X, \Delta, \ba^t) = \tau(X, \Delta, \ba^s)
\]
for all $s \in [t, t+\varepsilon]$.
\end{lemma}
\begin{proof}
In the case that the index of $K_X + \Delta$ is not divisible by $p > 0$, this is simply \cite[Lemma 3.23]{BlickleSchwedeTakagiZhang}.  But we may easily reduce to that case by using the fact that $\Tr^e(\tau(\omega_X, p^e(K_X + \Delta), \ba^{p^et})) = \tau(\omega_X, K_X + \Delta, \ba^t)$ from \autoref{lem.TransformationRulesForParameterTestModules}(b).
\end{proof}

We now prove another Lemma which shows that the test ideal is stabilized in the other direction.  This lemma is the key observation needed in the result of this section (compare with \cite[Proposition 6.3]{KatzmanLyubeznikZhangOnDiscretenessAndRationality} and \cite[Proposition 5.3]{BlickleSchwedeTakagiZhang}).  For a sketch of the proof see the  discussion after \autoref{thm.EasyEffectiveComputation}.  Most of the technicality of the following proof is designed to prove the surjectivity of the map labeled $\beta$ in that discussion.

\begin{lemma}
\label{lem.NoAccumulationFromBelow}
Suppose that $t = b/(p^c - 1)$.  Then there exists a number $\varepsilon > 0$ such that $\tau(\omega_X, \ba^s)$ is constant for all $t - \varepsilon < s < t$.
\end{lemma}
\begin{proof}
We may assume that $X = \Spec R$ is affine.
Let $\pi : Y \to X$ be the normalized blowup of $\ba$ and set $\ba \O_Y = \O_Y(-G)$.  We will find $a_0 > 0$ such that
\[
\tau(\omega_X, \ba^{t({p^a-1 \over p^a})})
\]
is constant for all $a \geq a_0$.  In fact, it is sufficient to do this for some sequence of $a$ going to infinity.

First, we notice that $\tau(\omega_Y, (t - \delta)G)$ is constant for $1 \gg \delta > 0$ by \autoref{lem.ParamTestModAreTestIdeals} and \cite[Proposition 5.3]{BlickleSchwedeTakagiZhang}.  Thus choose $d > 0$ such that
$t - \delta < t({p^{cd} - 1 \over p^{cd}}) < t$.  Now, by \autoref{lem.TransformationRulesForParameterTestModules}(b), we have surjections induced by trace for every $n > 0$
\[
T_{n} : F^{ncd}_* \Big(\tau(\omega_Y) \otimes \O_Y( -t(p^{ncd} - 1)G)\Big) \rightarrow \tau\big(\omega_Y, t({p^{ncd} - 1 \over p^{ncd}})G\big) = \tau\big(\omega_Y, (t-\delta)G\big).
\]
and also surjections
\[
\begin{array}{rcl}
W_n :&  F^{ncd}_* \Big( \tau(\omega_Y, (t-\delta)G) \otimes \O_Y(-t(p^{ncd} - 1)G) \Big) \\
= &  F^{ncd}_* \Big( \tau(\omega_Y, t({p^{cd} - 1 \over p^{cd}})G) \otimes \O_Y(-t(p^{ncd} - 1)G) \Big)\\
\to & \tau(\omega_Y, t({p^{cd} - 1 \over p^{cd + ncd}})G + t({p^{ncd} - 1 \over p^{ncd}})G)\\
= & \tau(\omega_Y, t({p^{cd} - 1 + p^{cd}(p^{ncd}-1) \over p^{cd+ncd} })G)\\
= & \tau(\omega_Y, t({p^{(1+n)cd} - 1  \over p^{(1+n)cd} })G)\\
= & \tau(\omega_Y, (t - \delta)G).
\end{array}
\]
Choose $e_0 > 0$ such that for all $k \geq t(p^{e_0c}-1)$ we have that the map $\pi_* (T_1 \otimes \O_Y(-kG))$ sending
\[
\pi_* F^{cd}_* \Big(\tau(\omega_Y) \otimes \O_Y( -t(p^{cd} - 1)G - p^{cd}k G)\Big) \to \pi_*\Big( \tau(\omega_Y, (t-\delta)G) \otimes \O_Y(-kG) \Big)
\]
is surjective, and also that the map $\pi_* (W_1 \otimes \O_Y(-t(p^{ec}-1)G))$ sending
\[
\pi_* F^{cd}_*\Big( \tau(\omega_Y, (t-\delta)G) \otimes \O_Y(-t(p^{cd} - 1)G - p^{cd}k G) \Big) \to \pi_* \Big( \tau(\omega_Y, (t-\delta)G) \otimes \O_Y(-kG) \Big)
\]
is surjective (both are possible since $-G$ is $\pi$-ample).  Since $W_n$ is simply $W_1$ composed with itself $n$-times (with appropriate twists by $\pi$-anti-ample divisors), we see that the map $W_{n,e}' := \pi_* (W_n \otimes \O_Y(-t(p^{ec}-1)G))$ is surjective as well for $e \geq e_0$.
\[
\xymatrix{
\vdots \ar@{->>}[d]\\
\pi_* F^{2cd}_*\Big( \tau(\omega_Y, (t-\delta)G) \otimes \O_Y(-t(p^{2cd} - 1)G - p^{2cd} t(p^{ec}-1)G) \Big) \ar@{->>}[d] \ar@/^15pc/@{->>}[dd]^{W_{2,e}'}\\
\pi_* F^{cd}_*\Big( \tau(\omega_Y, (t-\delta)G) \otimes \O_Y(-t(p^{cd} - 1)G - p^{cd} t(p^{ec}-1)G) \Big) \ar@{->>}[d]\\
\pi_* \Big(\tau(\omega_Y, (t - \delta)G) \otimes \O_Y(-t(p^{ec}-1)G)\Big)
}
\]
Likewise since $T_n$ is simply $T_1$ composed with $W_{n-1}$ (twisted appropriately), we see that $\beta = \pi_* (T_n \otimes \O_Y(-t(p^{ec}-1)G))$ is surjective for $e \geq e_0$.

Now we form the following composition where $\beta$ is surjective by the above analysis
\[
\begin{array}{rl}
& \pi_* F^{ncd+ec}_*  \Big(\tau(\omega_Y)\otimes \O_Y(-t(p^{ncd+ec} - 1)G)\Big) \\
\overset{\beta}{\twoheadrightarrow} & \pi_* F^{ec}_* \Big( \tau(\omega_Y, (t-\delta)G) \otimes \O_Y(-t(p^{ec}-1)G) \Big) \\
\xrightarrow{\alpha} & \omega_X .
\end{array}
\]
The image of this composition is equal to $\tau(\omega_X, \ba^{t({p^{ncd+ec}-1\over p^{ncd+ec}})})$ by \autoref{thm.ComputationOfTauSimple} at least for all $e \gg 0$.  However, this image is constant as we vary $n$ since the image of this composition is also the same as the image of $\alpha$.  Sending $n$ to infinity completes the proof.
\end{proof}

The rest of the argument for discreteness and rationality follows \cite{BlickleMustataSmithFThresholdsOfHypersurfaces,KatzmanLyubeznikZhangOnDiscretenessAndRationality}.

\begin{theorem}
\label{thm.DiscAndRat}
Suppose that $X$ is an $F$-finite normal scheme, $\Delta$ is a $\bQ$-divisor on $X$ such that $K_X + \Delta$ is $\bQ$-Cartier, and $\ba$ is an ideal sheaf on $X$.  Then the set of $F$-jumping numbers of $\tau(X, \Delta, \ba^t)$ is a discrete set of rational numbers.
\end{theorem}
\begin{proof}
First set $\eta : W \to X$ to be a finite cover such that $\eta^*(K_X + \Delta)$ is Cartier (see for example \cite[Lemma 4.5]{BlickleSchwedeTuckerTestAlterations}).  It follows from \autoref{lem.TransformationRulesForParameterTestModules}(a) that if we can show that the $F$-jumping numbers of $\tau(\omega_W, \eta^*(K_X + \Delta), (\ba \cdot \O_W)^t)$ are discrete and rational, then so are the $F$-jumping numbers of $\tau(\omega_X, K_X + \Delta, \ba^t) = \tau(X, \Delta, \ba^t)$.  However, since $\eta^*(K_X + \Delta)$ is Cartier, $\tau(\omega_W, \eta^*(K_X + \Delta), (\ba \cdot \O_W)^t) = \tau(\omega_W, (\ba \cdot \O_W)^t) \otimes \O_W(-\eta^*(K_X + \Delta))$, it is sufficient to show that the $F$-jumping numbers of $\tau(\omega_W, (\ba \cdot \O_W)^t)$ are discrete and rational.

The rest of the proof is now formally the same as \cite[Theorem 3.1]{KatzmanLyubeznikZhangOnDiscretenessAndRationality} by using \autoref{lem.NoAccumulationFromBelow}.  The point is that $p$ times a jumping number is a jumping number \cite[Lemma 3.25]{BlickleSchwedeTakagiZhang}, and that Skoda's theorem still holds for test ideals.
\end{proof}

\begin{corollary}
If $R = k\llbracket x_1, \dots, x_n \rrbracket$ and $\ba \subseteq R$ is any ideal and $k$ is an $F$-finite field, then the $F$-jumping numbers of $\tau(R, \ba^t)$ are a discrete set of rational numbers.
\end{corollary}

One might ask whether the $F$-finiteness hypothesis is really necessary.  For a not-necessarily $F$-finite, but excellent $\bQ$-Gorenstein local ring $(R, \bm)$ with a dualizing complex, the authors believe it can be shown that the $F$-jumping numbers of $\tau(R, f^t)$ are discrete and rational along the lines of \cite{KatzmanLyubeznikZhangOnDiscretenessAndRationality,BlickleSchwedeTakagiZhang}.  Thus we ask:

\begin{question}
Given an excellent $\bQ$-Gorenstein local ring with a dualizing complex, and an ideal $\ba \subseteq R$, is it true that the $F$-jumping numbers of $\tau(R, \ba^t)$ are discrete and rational?
\end{question}

\noindent
One cannot use the description of the test ideal given in this paper to tackle this question since that description critically uses $F$-finiteness.  However, one still may be able to blow up $\ba$ and apply local cohomology to the pushdown of appropriate sheaves so as to mimic the strategies applied herein.  This appears to be nontrivial however.

\section{Test ideals via alterations}

The descriptions of the test ideal above allow the main results of \cite{BlickleSchwedeTuckerTestAlterations} to be extended to triples $(X, \Delta, \ba^t)$ where $\ba$ is not principal in this section.  Throughout, we will largely assume $X$ is a normal variety over an $F$-finite field of characteristic $p > 0$ -- so that \cite{deJongAlterations} can be applied to find regular alterations.  However, we will also briefly discuss the (straightforward) generalizations to integral normal schemes in \autoref{rem.NonVarietyDescription1} and \autoref{rem.NonVarietyDescription2}.

Additionally, beginning in this section so as to avoid confusion, when writing the trace map we will generally write a subscript to indicate the morphism under consideration (since we are not only concerned with the Frobenius map).  So for example, what has previously been denoted above by $\Tr^e$ will be denoted by $\Tr_{F^e}$.  More generally, for an alteration $\rho : Y \to X$, recall once more that $\Tr_{\rho}$ is the Grothendieck trace as described in \cite[Proposition 2.18]{BlickleSchwedeTuckerTestAlterations}.

\begin{theorem}[Test ideals via alterations]
\label{thm.TestViaAlterations1}
Given log-$\Q$-Gorenstein triple $(X, \Delta, \ba^t)$ where $X$ is a variety over a perfect field, there exists a regular alteration
\[
\rho : W \to X
\]
with $\ba  \O_W = \O_W(-H)$
such that $\Tr_{\rho} \big(\rho_* \O_W( \lceil K_W - \rho^*(K_X + \Delta) - tH \rceil)\big) = \tau(X, \Delta, \ba^t)$.
It further follows that the test ideal $\tau(X, \Delta, \ba^t)$ is the intersection of these images over all alterations $\rho : W \to X$, \textit{i.e.}
\[
\tau(X, \Delta, \ba^t) = \bigcap_{\rho : W \to X} \Tr_{\rho} \big(\rho_* \O_W( \lceil K_W - \rho^*(K_X + \Delta) - tH \rceil)\big).
\]
\end{theorem}
\begin{proof}
Without loss of generality we may assume that $X$ is affine since any two alterations may be dominated by a third.
By \autoref{lem.TransformationRulesForParameterTestModules}(b) we know for $\Gamma = K_X + \Delta$ that
\[
\Tr_{F^b}\big( F^b_* \tau(X, p^b(K_X + \Delta) - K_X, \ba^{p^b t} \big)= \Tr_{F^b} \big(F^b_* \tau(\omega_X, p^b \Gamma, \ba^{p^b t})\big) = \tau(\omega_X, \Gamma, \ba^t) = \tau(X, \Delta, \ba^t)
\]
and so,
as in the start of the proof of \autoref{thm.EffectiveTauComputationGeneral1}, by choosing a sufficiently divisible $c > 0$ and replacing $\Delta$ by $ p^b(K_X + \Delta) - K_X$ and $t$ by $p^b t$, we may assume that $(p^c - 1)t$ is an integer and $(p^c - 1)(K_X + \Delta)$ is Cartier.

Set $\pi : Y \to X$ to be the normalized blowup of $\ba$ with $\ba  \O_Y = \O_Y(-G)$.  We then know by \autoref{thm.EffectiveTauComputationGeneral1} that there exists an $e = nc > 0$, such that
\[
\tau(X, \Delta, \ba^t) = \Tr_{F^e_X}\Bigg((F^e_X)_* \pi_* \Big(\tau(\omega_Y, \pi^*(K_X + \Delta) + tG) \otimes_Y \O_Y\big( (1-p^e)( \pi^* (K_X + \Delta) + tG)\big)\Big) \Bigg).
\]
By \autoref{lem.TransformationRulesForParameterTestModules} (which relies upon  \cite{deJongAlterations}), there exists $\eta : W \to Y$ a regular alteration such that
\[
\tau(\omega_Y, \pi^*(K_X + \Delta) + tG) = \Tr_{\eta} \Big(\eta_* \O_W\big(\lceil K_W - \eta^*\pi^*(K_X + \Delta) - t\eta^*G \rceil\big) \Big).
\]
Setting $H' = \eta^* G$ so that $\ba  \O_W = \O_W(-H')$, we observe that
\[
\begin{array}{rl}
  & \tau(\omega_Y, \pi^*(K_X + \Delta) + tG) \otimes_Y \O_Y\big( (1-p^e)( \pi^* (K_X + \Delta) + tG)\big) \\
= & \Tr_{\eta} \Big(\eta_* \O_W\big(\lceil K_W - \eta^*\pi^*(K_X + \Delta) - t H' - (p^e - 1) \eta^* \pi^* (K_X + \Delta) - (p^e - 1)tH' \rceil\big) \Big)\\
= & \Tr_{\eta} \Big(\eta_* \O_W\big(\lceil K_W - \eta^*\pi^*(F^e_X)^*(K_X + \Delta) - t p^e H' \rceil\big) \Big).
\end{array}
\]
Set $\rho = \eta \circ \pi \circ (F^e_X)$ and $H = p^e H'$.  Pushing forward by $\pi_*$ and applying $\Tr_{\pi \circ (F^e_X)}$ implies immediately that
\[
\tau(X, \Delta, \ba^t) \supseteq \Tr_{\rho} \big(\rho_* \O_W( \lceil K_W - \rho^*(K_X + \Delta) - tH \rceil)\big).
\]
Note we only get containment since we do not know that we still have a surjection after pushing forward by $\pi$ (we only have the surjection on $Y$ locally).

However, the reverse containment $\subseteq$ is the ``easy containment'' which always holds by the defining property of the test ideal, \cf the proof of \cite[Proposition 4.2]{BlickleSchwedeTuckerTestAlterations} or the first half of the proof of \autoref{prop.TauDescriptionViaBlowup}.  The intersection statement then follows immediately from the same ``easy containment'' as well.
\end{proof}

\begin{remark}
\label{rem.NonVarietyDescription1}
If $X$ is simply a normal $F$-finite scheme satisfying \autoref{conv.MainConvention}, instead of a variety, a variant of Theorem 6.1 still holds.  Indeed, there still exists a (not necessarily regular) alteration
$\rho : W \to X$
with $\ba \O_W = \O_W(-H)$ such that $\Tr_{\rho} \big(\rho_* \O_W( \lceil K_W - \rho^*(K_X + \Delta) - tH \rceil)\big) = \tau(X, \Delta, \ba^t)$.  The proof is unchanged, except that the use of \cite{deJongAlterations} is omitted.
\end{remark}

In \cite{SchwedeTuckerZhang}, the authors showed the stronger statement that there exists a single (regular) alteration allowing the computation of the test ideal even as the coefficient $t$ was allowed to vary (including when $t$ takes on irrational values).  The same methods go through, at least when combined with the Skoda type theorems we have developed in this paper, and so we obtain the following generalization.

\begin{theorem}
\label{thm.TestViaAlterations2}
With notation as above, there exists a regular alteration $\rho : W \to X$, $\ba  \O_W = \O_W(-H)$, independent of $t$ such that
\begin{equation}
\label{eq.ConstantTauForSingleAlteration}
\Tr_{\rho} \big(\rho_* \O_W( \lceil K_W - \rho^*(K_X + \Delta) - tH \rceil)\big) = \tau(X, \Delta, \ba^t).
\end{equation}
for all $t \in \bR$, $t > 0$.
\end{theorem}
\begin{proof}
It is sufficient to prove the theorem for $X$ affine.
Essentially following the same argument as in \cite[Lemma 4.1, Theorem 4.2]{SchwedeTuckerZhang}, we begin with the following claim.

\begin{claim}
\label{clm.AdjacentJumpingNumbersConstant}
For any two adjacent $F$-jumping numbers, $t_0 < t_1$, there is an alteration $\rho : W \to X$ such that \autoref{eq.ConstantTauForSingleAlteration} holds for all $t_0 \leq t < t_1$.
\end{claim}

\begin{proof}[Proof of claim] (\textit{cf.} \cite[Lemma 4.1]{SchwedeTuckerZhang})
We know that talking about adjacent $F$-jumping numbers make sense because the set of $F$-jumping numbers have no limit points.  We also know that $t_0, t_1 \in \bQ$.  For varieties, this is described in \cite[Remark 3.4]{SchwedeTuckerZhang} but it also of course follows from \autoref{thm.DiscAndRat}.  If one wants to work with schemes as in \autoref{rem.NonVarietyDescription2} below, one truly needs \autoref{thm.DiscAndRat}.

Regardless, for some $\mu : W \to X$, we have
\[
\tau(X, \Delta, \ba^{t_0}) = \Tr_{\mu} \big(\mu_* \O_W( \lceil K_W - \mu^*(K_X + \Delta) - t_0 H \rceil)\big)
\]
where $\ba \cdot \O_W = \O_W(-H)$.
Fix $t'$ such that $t < t' < t_1$.  Then we know for some further $\mu' : W' \to X$ we have
\[
\tau(X, \Delta, \ba^{t'}) = \Tr_{\mu'} \big(\mu'_* \O_{W'}( \lceil K_{W'} - \mu'^*(K_X + \Delta) - t' H' \rceil)\big)
\]
where $\ba \cdot \O_{W'} = \O_{W'}(-H')$.
Now, we simply observe that
\[
\begin{array}{rl}
  & \tau(X, \Delta, \ba^t)\\
=  & \tau(X, \Delta, \ba^{t_0})\\
= & \Tr_{\mu} \big(\mu_* \O_W( \lceil K_W - \mu^*(K_X + \Delta) - t_0 H \rceil)\big) \\
\supseteq & \Tr_{\mu} \big(\mu_* \O_W( \lceil K_W - \mu^*(K_X + \Delta) - t H \rceil)\big) \\
\supseteq & \Tr_{\mu} \big(\mu_* \O_W( \lceil K_W - \mu^*(K_X + \Delta) - t' H \rceil)\big) \\
\supseteq & \Tr_{\mu'} \big(\mu'_* \O_{W'}( \lceil K_{W'} - \mu'^*(K_X + \Delta) - t' H' \rceil)\big) \\
= & \tau(X, \Delta, \ba^{t'})\\
= & \tau(X, \Delta, \ba^t)
\end{array}
\]
Thus all the terms are equal and the first claim is proven.  We also note again that, if we do not require $W$ to be regular, we may assume that $W$ is a finite cover of  the normalized blowup $Y$ of $\ba$ by \cite[Theorem 4.6]{BlickleSchwedeTuckerTestAlterations}.
\end{proof}
Theorem \ref{eq.ConstantTauForSingleAlteration} now follows for bounded values of $t$, since then there are only finitely many jumping numbers that need be considered.  By the claim, we may pick a common alteration that works for all of them.
Next we attack the case of unbounded $t$.  The strategy is roughly to use Skoda's theorem on both sides of equation \autoref{eq.ConstantTauForSingleAlteration}.

First, by Skoda's theorem for the test ideal \cite[Theorem 4.1]{HaraTakagiOnAGeneralizationOfTestIdeals}, for any $t$ larger than $r$, the number of generators of $\ba$, we have $\tau(X, \Delta, \ba^t) = \tau(X, \Delta, \ba^{r - 1 + \{ t \}}) \cdot \ba^{\lfloor t \rfloor - r + 1}$.  We need the same result for the pushforward from the alteration as well.  In particular, we need a variant of Skoda's theorem for the pushforward.  Note that we must do without a test element divisor $D$ that appeared in the discussion before \autoref{prop.BasicSkoda}.  We accomplish this below -- but on a possibly non-regular alteration.

\begin{claim}
\label{claim:skodapushforwardanalog}
There exists a sufficiently large alteration $\rho : W \to X$, satisfying \autoref{eq.ConstantTauForSingleAlteration} for $0 < t \leq r + 1$  (where $r$ is the number of generators of $\ba$), and such that instead for all $t \geq r + 1$ we have
\[
\ba \cdot \Tr_{\rho} \Big(\rho_* \O_W( \lceil K_W - \rho^*(K_X + \Delta) - (t - 1)H \rceil) \Big) = \Tr_{\rho} \Big(\rho_* \O_W( \lceil K_W - \rho^*(K_X + \Delta) - tH \rceil) \Big).
\]
\end{claim}
\begin{proof}[Proof of claim]
Choose $\rho : W \to X$ to be an alteration such that \autoref{eq.ConstantTauForSingleAlteration} holds for all $0 < t \leq r + 1$.  Observe that by \cite[Theorem 4.6]{BlickleSchwedeTuckerTestAlterations} and the proof of \autoref{clm.AdjacentJumpingNumbersConstant}, we can assume that $\rho : W \to X$ is a finite cover of the normalized blowup $Y$ of $\ba$.  In particular, we can assume that $-H$ is $\rho$-ample.  Additionally we can assume that if $W \xrightarrow{\mu} Z \xrightarrow{\nu} X$ is the Stein factorization of $\rho$, then $\nu^*(K_X + \Delta)$ is an integral Cartier divisor.

Fix $s_1, \dots, s_r \in \ba$ to be our generators.
It follows that $\pi^{*} s_{1}, \ldots, \pi^{*}s_{r}$ are globally generating sections of $\O_{W}(-H)$.  This implies $(\pi^{*}s_{1})^{p^{e}}, \ldots, (\pi^{*}s_{r})^{p^{e}}$ globally generate $\O_{W}(-p^{e}H)$ for any $e > 0$, and we may form the corresponding Koszul complex
\begin{equation}
\label{eq.KoszulComplexInTauSingleAlteration}
0 \to \sF_{r} \to \sF_{r-1} \to \cdots \sF_{1} \to \sF_{0} \to 0
\end{equation}
where $\sF_{i} = \O_{W}(ip^{e}H)^{\oplus {r \choose i}}$ and each of the maps are essentially given (up to sign) as multiplication by the sections $(\rho^{*}s_{j})^{p^{e}}$. Since this complex is a locally free resolution of the (flat) sheaf $\sF_{0} = \O_{W}$, this complex remains exact after tensoring by any quasicoherent sheaf on $W$.

When $t  \geq r+1$ and we tensor by $\O_{W}( \lceil K_{W} - p^{e}(\rho^{*}(K_{X} + \Delta) + t H) \rceil)$, the $i$-th entry in the complex becomes $\sG_{i} = \O_{W}( \lceil K_{W} -p^{e}(\pi^{*}(K_{X}+\Delta) + (t - i)H) \rceil)^{\oplus {r \choose i}}$.  Now we consider $\lceil -\lambda H\rceil$ as $\lambda$ varies and observe that there are only finitely many divisors $\lceil -\{ \lambda \} H \rceil$.  Therefore, since $\lceil -\lambda H\rceil = \lceil -\{\lambda\} H \rceil - \lfloor \lambda \rfloor H$ (and so takes finitely many values up to twisting by relatively ample divisors), and since $\rho^*(K_X + \Delta)$ is pulled back from a Cartier divisor on $Z$ (the Stein factorization of $W$), we observe for $e \gg 0$ (which we can pick independently of $t$) that by Serre vanishing and the projection formula, $R^{j} \rho_{*} \sG_{i} = 0$ for all $i$ and any $j > 0$ since $t - r \geq 1$.  Hence our complex \autoref{eq.KoszulComplexInTauSingleAlteration} remains exact after applying $\rho_{*}( \blank)$, and as $F_{*}^{e}(\blank)$ is exact (since $F$ is finite) we have that the complex
\[
0 \to F^{e}_{*}\rho_{*} \sG_{r} \to F^{e}_{*}\rho_{*} \sG_{r-1} \to \cdots \to F^{e}_{*}\rho_{*} \sG_{1} \to F^{e}_{*}\rho_{*} \sG_{0} \to 0
\]
is exact on $X$.  Furthermore, after having applied $F^{e}_{*}( \blank)$, we may view the arrows as given by multiplying by $s_{1}, \ldots, s_{r}$.  Taking images under $\Tr_{\rho}$ and $\Tr_{F^e_X}$ preserves exactness on the right, giving a surjection for $e \gg 0$ (again, independent of $t$)
\[
\begin{array}{rl}
  & \Tr_{F^e_X}\Big(F^e_* \Tr_{\rho}\Big( \rho_* \O_W\big(\lceil K_W - (F^e)^*\rho^*(K_X + \Delta) - t (F^e)^* H \rceil\big) \Big) \Big)\\
= & \ba \cdot \Tr_{F^e_X}\Big(F^e_* \Tr_{\rho}\Big( \rho_* \O_W\big(\lceil K_W - (F^e)^*\rho^*(K_X + \Delta) - (t - 1) (F^e)^* H \rceil\big) \Big) \Big).
\end{array}
\]
Now replacing $\rho$ by $\rho \circ F^e$ (again, observing that we have successfully chosen $e$ independent of $t$), we have proven our claim.
\end{proof}

Let us now return to the proof of Theorem \ref{thm.TestViaAlterations2}. Note that the containment $\supseteq$ of \autoref{eq.ConstantTauForSingleAlteration} always holds, and that further alterations can only shrink the image.
Therefore, if we can prove that \autoref{eq.ConstantTauForSingleAlteration} holds for any alteration, then it also holds for some regular alteration as well.  To that end, let $\rho : W \to X$ be the alteration constructed in Claim \ref{claim:skodapushforwardanalog}. For any $t \geq r+1$, we have that
\[
\begin{array}{rl}
& \Tr_{\rho}\Big( \rho_* \O_W\big(\lceil K_W - \rho^*(K_X + \Delta) - t H \rceil\big) \Big) \\
= & \ba^{\lfloor t \rfloor - r + 1} \cdot \Tr_{\rho}\Big( \rho_* \O_W\big(\lceil K_W - \rho^*(K_X + \Delta) - (r - 1 + \{ t \}) H \rceil\big) \Big)\\
= & \ba^{\lfloor t \rfloor - r + 1} \cdot \tau(X, \Delta, \ba^{r - 1 + \{ t \}}) \\
= & \tau(X, \Delta, \ba^t).
\end{array}
\]
This completes the proof of Theorem \ref{thm.TestViaAlterations2}.
\end{proof}

\begin{remark}
\label{rem.NonVarietyDescription2}
As before, this result still holds for normal $F$-finite integral schemes satisfying \autoref{conv.MainConvention}, so long as one gives up the requirement that the alteration be regular.
\end{remark}

In \cite{BlickleSchwedeTuckerTestAlterations,SchwedeTuckerZhang}, when $\ba$ is principal, we showed that the alterations chosen could be chosen to be generically separable.  Thus, we are led ask the following question.

\begin{question}
Can one always choose a separable alteration for \autoref{thm.TestViaAlterations2} or even for \autoref{thm.TestViaAlterations1}?
\end{question}

\section{Vector subspaces of global sections}

We now define certain global variants of the test ideal as subspaces of sections of line bundles.  Indeed, if $L$ is very ample inducing a projectively normal embedding $X \subseteq \bP^n_k$, we shall see that these subspaces carry the information of various graded pieces of the test ideal on the affine cone over $X$ (see for example \cite{SmithFujitaFreenessForVeryAmple} and \cite{HaraACharacteristicPAnalogOfMultiplierIdealsAndApplications}).  However, they apply even without ampleness assumptions on the line bundles.  As mentioned in the introduction, these subspaces first appeared in unpublished work of C.~Hacon and the first author.

\begin{definition}
\label{def.P0Definition}
Suppose that $(X, \Delta, \ba^t)$ is a triple where $X$ is a proper variety and that $L$ is a Cartier divisor on $X$.  Then we define
\[
P^0\big(X, \O_X(L) \tensor \tau(X, \Delta, \ba^t)\big)
\]
to be
\[
\displaystyle
\bigcap_{\substack{\pi : Y \shortrightarrow X}} \bigcap_{D \subseteq Y} \bigcap_{e_{0}\geq 0}  \left( \sum_{e \geq e_0} \Tr_{\pi \circ F^e} \Big(H^0\big(Y, \Big( F^e_* \O_Y(\lceil K_Y - p^e \pi^*(K_X + \Delta + tG_Y - L)  -  D \rceil)  \big)\Big)\Big)\right)
\]
where $\pi$ ranges over all proper birational maps such that $Y$ is normal with $\ba  \O_Y = \O_Y(-G_Y)$ invertible,  and $D$ runs over all effective divisors on $Y$.
So as to justify the notation somewhat, we remark that this is a subspace of the vector space
\[
H^0\big(X, \O_X(L) \tensor \tau(X, \Delta, \ba^t)\big) \subseteq H^0(X, \O_X(L) ).
\]
\end{definition}

\begin{lemma}
\label{lem.StabilizingImageForP0}
The intersection in \autoref{def.P0Definition} stabilizes.  In particular, there is a choice of $\pi$, $D$, and $e_0 \geq 0$, such that
\[
\begin{array}{rl}
& P^0\big(X, \O_X(L) \tensor \tau(X, \Delta, \ba^t)\big)\\
 = & \sum_{e \geq e_0} \Tr_{\pi \circ F^e} \Big(H^0\big(Y, \Big( F^e_* \O_Y(\lceil K_Y - p^e \pi^*(K_X + \Delta + tG_Y - L)  - D \rceil)  \big)\Big)\Big).
 \end{array}
\]
\end{lemma}
\begin{proof}  Observe that we are taking this intersection in a finite dimensional vector space,  simplifying things considerably.  The idea is simply to peel off each intersection one by one in the following order.
\begin{itemize}
\item[(a)]  We will argue below that passing to a further birational map $\eta : Z \to Y$ can only shrink the images in question. Granting this, we can fix a $\pi$ sufficiently large so that the image stabilizes.
\item[(b)]  For fixed $\pi : Y \to X$, choosing a larger $D$ obviously can only shrink the images.  Thus, we can find $D$ sufficiently large so that the image stabilizes.
\item[(c)]  For fixed $\pi$ and $D$, choosing larger $e_0$ can again obviously only shrink the image.  Thus, we can find $e_0$ sufficiently large so that the image stabilizes.
\end{itemize}

Thus, to verify the lemma, it remains only to justify  the choice of $\pi : Y \to X$ in (a).  Indeed, consider a further birational map $\eta : Z \to Y$ and the factorization
\[
\underbrace{Z \xrightarrow{\eta} Y \xrightarrow{\pi} X }_{\gamma}
\]
and fix a divisor $D_Y$ on $Y$.  By \autoref{lem.UniformImages}, there exists a divisor $D_Z$ on $Z$ such that we have an inclusion
\begin{equation*}
\label{eq.ClaimPushforwardContainment}
\begin{array}{rl}
& \eta_* \O_Z(\lceil K_Z - p^e \gamma^*(K_X + \Delta + tG_Y - L) - D_Z \rceil) \\
\subseteq & \O_Y(\lceil K_Y - p^e \pi^*(K_X + \Delta + tG_Y - L)  - D_Y \rceil)
\end{array}
\end{equation*}
for all $e$.  But now we see that choosing further $Z \to Y \to X$ only shrinks the images (after intersecting over all divisors $D$) as claimed, and so the Lemma is proven.
\end{proof}

The next proposition says that in fact we never have to take a $Y$ larger than the normalized blow-up of $\ba$ in order to compute $P^{0}$. Of course, if one prefers to use a different birational model instead (\textit{e.g.} a log resolution at hand), one is certainly still free to do so.

\begin{proposition}
The vector space $P^0 \big(X, \O_X(L) \tensor \tau(X, \Delta, \ba^t) \big)$ can be computed just from $\pi : W \to X$ where $W$ is the normalized blowup of $\ba$.  In particular, the subspace $P^0 \big(X, \O_X(L) \tensor \tau(X, \Delta, \ba^t) \big)$ is equal to
\medskip
\begin{equation}
\label{eq.StabilizeOnNormalBlowup}
\displaystyle
\bigcap_{D \geq 0} \bigcap_{e_0 \geq 0} \left( \sum_{e \geq e_0} \Tr_{\pi \circ F^e} \Big(H^0\big(W, \Big( F^e_* \O_W(\lceil K_W - p^e \pi^*(K_X + \Delta + tG_W - L)  - D \rceil)  \big)\Big)\Big)\right)
\end{equation}
where $D$ runs over all effective divisors on $W$.
\end{proposition}
\begin{proof}
Suppose that $\gamma : Y \to X$ is a projective birational map with $Y$ normal such that $\ba \O_Y = \O_Y(-G_Y)$ is invertible.  We may also assume that $\gamma : Y \to X$ can be used to stabilize the intersection \autoref{eq.StabilizeOnNormalBlowup}, together with the divisor $E \geq 0$ on $Y$ as in Lemma \ref{lem.StabilizingImageForP0}.   Then $\gamma$ factors through $\pi$ by the universal property of blowing up, so write
\[
\underbrace{Y \xrightarrow{\eta} W \xrightarrow{\pi} X}_{\gamma}.
\]
By \autoref{lem.UniformImages}, there exists an effective divisor $D$ on $W$  such that
\[
\begin{array}{rl}
& \O_W(\lceil K_W - p^e \pi^*(K_X + \Delta + tG_W - L) -  D \rceil) \\
\subseteq & \eta_* \O_Y(\lceil K_Y  - p^e \gamma^*(K_X + \Delta + tG_Y - L)  - E \rceil)
\end{array}
\]
for all $e \geq 0$.
Making $D$ larger is harmless, and so we may further assume that $D$ can be used to stabilize the intersection in \autoref{eq.StabilizeOnNormalBlowup}.

Thus for all $e \geq 0$ we have the factorization
\[
\begin{array}{rl}
& \pi_* \O_W(\lceil K_W - p^e \pi^*(K_X + \Delta + tG_W - L)  - D \rceil) \\
\hookrightarrow & \gamma_* \O_Y(\lceil K_W - p^e \pi^*(K_X + \Delta + tG_W - L)  - E \rceil) \\
\xrightarrow{\Tr_{\gamma}} & \O_X(L).\end{array}
\]
But then it follows that we have containment of the images
\[
\begin{array}{rl}
& \Tr_{\pi \circ F^e} \Big(H^0\big(W, \Big( F^e_* \O_W(\lceil K_W - p^e \pi^*(K_X + \Delta + tG_W - L)  - D \rceil)  \big)\Big)\Big)\\
 \subseteq & \Tr_{\gamma\circ F^e} \Big(H^0\big(Y, \Big( F^e_* \O_Y(\lceil K_W - p^e \pi^*(K_X + \Delta + tG_W - L)  - E \rceil)  \big)\Big)\Big)
\end{array}
\]
which gives that the intersection in \autoref{eq.StabilizeOnNormalBlowup} is contained in $P^0 \big(X, \O_X(L) \tensor \tau(X, \Delta, \ba^t) \big)$.  Observing that the reverse containment is obvious from point (a) in the proof of \autoref{lem.StabilizingImageForP0} above completes the proof.
\end{proof}

Note that the definition of $P^{0}$ has non-trivial content even in comparatively simple settings, \textit{e.g.} when there is no ideal sheaf under consideration. For sufficiently ample adjoint line bundles on smooth varieties, we now show that $P^0$ simply agrees with $H^0$.

\begin{lemma}
\label{lem.PNotIsH0Smooth}
Suppose that $X$ is a smooth projective variety, and $\sL$ is an ample line bundle on $X$.  Then $P^0(X, \omega_X \tensor \sL^n) := P^0(X, \tau(X, -K_X) \tensor \sL^n) = H^0(X, \omega_X \tensor \sL^n)$ for all $n \gg 0$.
\end{lemma}
\begin{proof}
We set $S = \oplus_{i \geq 0} H^0(X, \sL^i)$ to be the section ring with respect to $\sL$.  Then $S$ is normal with an isolated singularity at the origin $S_+ = \oplus_{i \geq 1}  H^0(X, \sL^i)$.  For any divisor $D \geq 0$ on $X$, we may find $s \in H^0(X, \sL^j)$ such that $D \leq \Div(s)$ for some $j > 0$.  Therefore, when forming $P^0$, we may take $D$ only of the form $V_X(s)$ for $s$ as above.  Then it follows immediately that
\[
\tau(\omega_S) = \bigoplus_i P^0(X, \omega_X \tensor \sL^i)
\]
from the fact we may choose $s$ to be a test element.  We claim that for all $n \gg 0$,
\[
[\tau(\omega_S)]_n = [\omega_S]_n = H^0(X, \omega_X \tensor \sL^{n})
\]
where $[\blank]_n$ denotes the $n$th graded piece of a module.  But this is trivial since $\tau(\omega_S) = \omega_S$ away from the origin since that is where the singularity is located.

In fact, one sees that requiring that $X$ is smooth is more than required.  It is more than sufficient to merely require that $X$ has $F$-rational singularities (which just means that $\tau(\omega_X) = \omega_X$ and that $X$ is Cohen-Macaulay \cite[Section 8.1]{SchwedeTuckerTestIdealSurvey}).
\end{proof}

\begin{remark}
We believe (but do not verify herein) that more general equalities hold as well for $\tau(X, \Delta, \ba^n) \tensor \O_X(nH)$ for sufficiently ample $H$ (\textit{e.g.} when $\pi^* H - G$ is ample).
\end{remark}

Similar subspaces of sections to $P^0$ have also recently appeared, at least in case we take $\ba = \O_{X}$.  We conclude this section with some comparisons of these subspaces.

\begin{definition}
\label{def.S0Definition}
Suppose that $\Delta \geq 0$ is such that $K_X + \Delta$ is $\bQ$-Cartier with index not divisible by $p$ (in particular $(p^{e_1} - 1)(K_X + \Delta)$ is Cartier for some $e_1$).  Then for any line bundle $L$ on $X$ and any $e > 0$ such that $e_1 | e$, we have a map
$F^e_* \O_X(K_X - (p^e -1) \Delta) \to \O_X(K_X)$
and thus obtain (no roundings are necessary)
\[
\Tr^e = \Tr_{F^e} : F^e_* \big(\O_X((1-p^e)(K_X + \Delta)) \tensor L^{p^e}\big) \to L.
\]
Since the map $\Tr^{(n+1)e_1}$ factors through $\Tr^{ne}$ for any integer $sn$, the images on global sections stabilize for $e = ne_1 \gg 0$, and
we define
\[
S^0(X, \sigma(X, \Delta) \tensor L) = \Tr_{F^e}\big(H^0(X, F^e_* \O_X( (1-p^e)(K_X + \Delta)) \tensor L^{p^e} )\big) \subseteq H^0(X, L)
\]
for this stable image when $e = ne_1 \gg 0$.
\end{definition}

\begin{proposition}[$P^0$ versus $S^0$]
\label{prop.P0vsS0}
With notation as above, we have the containment
\[
P^0(X, \tau(X, \Delta) \tensor \O_X(L) ) \subseteq S^0(X, \sigma(X, \Delta) \tensor \O_X(L) ).
\]
Additionally, fix a Cartier divisor $D$ as in \autoref{claim.TauDescriptionViaBlowup} and also stabilizing $P^0(X, \tau(X, \Delta) \tensor L)$.  Then for all rational $\varepsilon > 0$ without $p$ in its denominator, we have
\[
S^0(X, \sigma(X, \Delta + \varepsilon D) \tensor \O_X(L) ) \subseteq P^0(X, \tau(X, \Delta) \tensor \O_X(L)).
\]
\end{proposition}

\begin{proof}
The first containment is obvious except for the fact that we are only summing over all $e \gg 0$ in $P^0$ and we are summing over $ne_1 \gg 0$ to compute $S^0$.  We explain how to overcome this difficulty.  Note that there is nothing to blow up.  We fix $D$ computing $P^0$ for $\pi = \id_X$. 
Suppose that $n e_1 \gg 0$ can be used to stabilized $S^0$.  Chose $e > n e_1$ 
and note that $(p^e - 1)(K_X + \Delta)$ is not necessarily integral. 

It then follows from the fact that $(p^{ne_1} - 1)p^{e - ne_1} \leq p^e$ that
\[
\begin{array}{rl}
& \Tr_{F^{e - n e_1}}\Big(F^{e - ne_1}_* \O_X(\lceil K_X - p^e(K_X + \Delta) - D \rceil) \Big)\\
= & \Tr_{F^{e - n e_1}}\Big(F^{e - ne_1}_* \O_X(K_X - p^e K_X - \lfloor p^e \Delta + D \rfloor) \Big)\\
\subseteq & \Tr_{F^{e - n e_1}}\Big(F^{e - ne_1}_* \O_X(K_X - p^{e - ne_1} p^{ne_1} K_X - (p^{ne_1} - 1)p^{e - ne_1} \Delta ) \Big)\\
\subseteq & \O_X\big(K_X - p^{ne_1} K_X - (p^{ne_1} - 1) \Delta \big)
\end{array}
\]
and the first containment follows by applying $\Tr_{F^{ne_1}}(F^{ne_1}_* \blank)$.
The second containment is straightforward to check directly.
\end{proof}

\begin{remark}
In many cases $S^0(X, \sigma(X, \Delta + \varepsilon D) \tensor \O_X(L) ) \neq 0$, \cf \cite{SchwedeACanonicalLinearSystem,MustataNonNefLocusPositiveChar,MustataSchwedeSeshadri,CasciniHaconMustataSchwedeNumDimPseudoEffective}.  In particular, if $D$ is as in the statement of \autoref{prop.P0vsS0} and $S^0(X, \sigma(X, \Delta + \varepsilon D) \tensor \O_X(L) )$ is non-vanishing for some $\varepsilon > 0$, then $P^0(X, \tau(X, \Delta) \tensor \O_X(L))$ is also non-vanishing.  See also Section 9 below for a number of non-vanishing statements for $P^{0}$ on curves.
\end{remark}

Another subspace of sections $T^0$ was also introduced in \cite{BlickleSchwedeTuckerTestAlterations}.  This variant roughly coincides with $S^0$ except the intersection is taken over all finite covers (rather than simply iterates of Frobenius), and we refer to \textit{loc. cit.} for the precise formulation.  It would seem natural to ponder the relationship of $T_{0}$ and $P_{0}$ as well.
\begin{question}
Is it true that $T^0(X, \tau(X, \Delta) \tensor \O_X(L)) = P^0(X, \tau(X, \Delta) \tensor \O_X(L))$ for all divisors $L$?  What about for ample $L$?
\end{question}
\noindent
This question is closely related to the question of whether splinters are strongly $F$-regular and also to the comparison of weak versus strong $F$-regularity, two of the main open problems in the theory of tight closure, \cf \cite{SinghQGorensteinSplinters,LyubeznikSmithStrongWeakFregularityEquivalentforGraded,LyubeznikSmithCommutationOfTestIdealWithLocalization}.


\section{Global division theorem: general case}

In this section, we show the last of the main results of this paper mentioned in the introduction: the global division theorem for $P^0$.


\begin{theorem}[Global division theorem]
\label{thm.MainResult}
Suppose that $X$ is a normal projective $n$-dimensional variety over an algebraically closed field of characteristic $p > 0$, and that $\Delta \geq 0$ is a $\bQ$-divisor such that $K_X + \Delta$ is $\bQ$-Cartier.  Suppose that $\ba \subseteq \O_X$ is an ideal sheaf and $L$ is a Cartier divisor such that $\O_X(L) \tensor \bc$ is globally generated by sections $s_1, \dots, s_r \in \Gamma(X, \O_X(L) \tensor \bc)$ for some reduction $\bc$ of $\ba$ (by replacing the $s_i$ with general linear combinations, we may always assume that $r \leq n+1$).    Fix $M$ a Cartier divisor on $X$ such that $M - K_X - \Delta$ is nef and big, and a positive integer $m \geq r$.  Then any section
\[
s \in P^0\Big( X, \O_X(M + m L) \tensor \tau(X, \Delta, \ba^m) \Big)
\]
can be expressed as a linear combination
\[
s = \sum h_i s_i
\]
with $h_i \in P^0\Big( X, \O_X(M + (m-1)L) \tensor \tau(X, \Delta, \ba^{m-1}) \Big)$.
\end{theorem}

\begin{remark}
Note that, in the case that $\ba = \O_X$ and that $L$ is a globally generated ample divisor (by $r = n+1$ sections), Theorem~\ref{thm.MainResult} implies that the subset of global sections $P^0\Big( X, \O_X(M + m L) \tensor \tau(X, \Delta) \Big)$ globally generates $\O_X(M + m L) \tensor \tau(X, \Delta)$ for $m \geq n$, \cf \cite{KeelerFujita,SchwedeACanonicalLinearSystem,MustataNonNefLocusPositiveChar}.  To see this, note that a mild generalization of \autoref{lem.PNotIsH0Smooth} implies that $\O_X(M + m L) \tensor \tau(X, \Delta)$ is globally generated for $m \gg 0$.  Then one can repeatedly use \autoref{thm.MainResult} to prove it for $m-1,m-2,\ldots$ and so on (until reaching $n$ itself).
\end{remark}

\begin{proof}
Fix $\pi \: Y \to X$ to be the normalized blowup of $X$ along $\ba$, so that $\ba \O_{Y} = \O_{Y}(-G)$ for some effective Cartier divisor $G$ on $Y$.

First observe that as $\pi^{*}L - G$ is globally generated, it is certainly nef.  Also observe that since $M - K_X - \Delta$ is big and nef, there exists some divisor $D_1 \geq 0$ on $X$ and integer $n > 0$ such that $M_1 = n(M - K_X - \Delta)$ is Cartier and $M_1 - D_1 = n(M - K_X - \Delta) - D_1$ is ample.  Next, choose a $\pi$-antiample divisor $D_{2} \geq 0$ on $Y$ and integer $m > 0$ such that $m\pi^{*}(M_{1} - D_{1}) - D_{2}$ is ample.
Choose also a divisor $E \geq 0$ on $Y$ sufficiently large so as to stabilize the intersections \eqref{eq.StabilizeOnNormalBlowup} computing $P^0(X, \O_{X}(M + mL)\tensor \tau(X, \Delta, \ba^{m}))$ and $P^0(X, \O_{X}(M + (m-1)L)\tensor \tau(X, \Delta, \ba^{m-1}))$. Finally, we fix $l > 0$ large enough so that $l(m\pi^{*}(M_{1} - D_{1})-D_{2})-E$ is ample, and also that $-lm\pi^{*}D_{1}-lD_{2}-E$ is $\pi$-ample.

Put $D' = lm\pi^{*}D_{1} + l D_{2} + E$.  To briefly summarize what is needed from the above discussion, we have
\begin{itemize}
\item[(i)]  $\pi^{*}L - G$ is a nef Cartier divisor on $Y$,
\item[(ii)] $M_{1} = n(M - K_{X} - \Delta)$ is a nef Cartier divisor on $X$, and
\item[(iii)] $D'$ is an effective $\pi$-antiample divisor on $Y$ so that $lm \pi^{*}M_{1} - D'$ is ample.
\end{itemize}
For every $e > 0$, write $p^e = \lfloor p^e/n \rfloor n + r_e$ with $0 \leq r_e < n$.
We now choose $k$ sufficiently large so that for all $e \gg 0$ (whenever $p^e /n > lmk$) and all $i > 0$, we have the vanishing
\begin{equation}
\label{eq.GlobalDivFujitaGlobalVan}
{
\small
\begin{array}{rl}
& H^i\big(Y, \O_Y(\lceil K_{Y} + jp^{e} (\pi^{*}L - G) + p^{e} \pi^{*} (M - K_X - \Delta) - kD'\rceil)\big) \\
= & H^i\big(Y, \O_Y(\lceil K_{Y} + jp^{e} (\pi^{*}L - G) + (p^{e}/n) \pi^{*} M_1 - kD'\rceil)\big) \\
= & H^i\big(Y, \O_Y(\lceil K_{Y} + jp^{e} (\pi^{*}L - G) + (r_e/n) \pi^* M_1 + \lfloor p^{e}/n\rfloor \pi^{*} M_1 - kD'\rceil)\big) \\
= & H^i\big(Y, \O_Y(K_{Y} + jp^{e} (\pi^{*}L - G) + \lceil (r_e/n) \pi^* M_1 \rceil + (\lfloor p^{e}/n \rfloor  - lmk) \pi^{*} M_1 + k(lm \pi^*M_1 - D') )\big)\\
= & 0
\end{array}
}
\end{equation}
by Fujita's vanishing theorem \cite{FujitaVanishingTheoremsForSemiPositive}, where $j$ varies from $0$ to $r$.  Let us explain this briefly in slightly more detail.
Note that in the last line of \autoref{eq.GlobalDivFujitaGlobalVan}, the rounded term $\lceil (r_e/n) \pi^* M_1 \rceil$ takes on only finitely many values so Fujita's vanishing applies using the ampleness from (iii) above together with the nefness from (i) and (ii).
Similarly, by making $k$ larger if necessary, we can ensure for all $i > 0$
\begin{equation}
\label{eq.GlobalDivFujitaLocalVan}
\begin{array}{rl}
0 = & R^i \pi_* \O_Y(\lceil K_{Y} + jp^{e} (\pi^{*}L - G) + p^{e} \pi^{*} (M - K_X - \Delta) - k D'\rceil)
\end{array}
\end{equation}
by relative Fujita's vanishing \cite[Theorem 1.5]{KeelerAmpleFiltersOfInvertibleSheaves}, where again $j$ varies from $0$ to $r$.  Put $D = kD'$.

The sections $\pi^{*} s_{1}, \ldots, \pi^{*}s_{r}$ globally generate the sheaf $\O_{Y}(\pi^{*}L-G)$, hence we also know $(\pi^{*}s_{1})^{p^{e}}, \ldots, (\pi^{*}s_{r})^{p^{e}}$ globally generate $\O_{Y}(p^{e}(\pi^{*}L -G))$ for any $e > 0$. Form the corresponding Koszul complex
\[
0 \to \sF_{r} \to \sF_{r-1} \to \cdots \sF_{1} \to \sF_{0} \to 0
\]
where $\sF_{j} = \O_{Y}(-jp^{e}(\pi^{*}L-G))^{\oplus {r \choose j}}$ and each of the maps is (up to sign) essentially given as multiplication by the sections $(\pi^{*}s_{j})^{p^{e}}$.  Observe that this complex is locally exact by \cite[Theorem 1.6.5]{BrunsHerzog}, and hence is a locally free resolution of the (flat) sheaf $\sF_{0} = \O_{Y}$.  In particular, this complex remains exact after tensoring by any quasicoherent sheaf on $Y$.  Set
\[
\Lambda_{j} = \lceil K_{Y} + jp^{e} (\pi^{*}L - G) + p^{e} \pi^{*} (M - K_X - \Delta) - D\rceil
\]
for $j = 0, \ldots, m$. After we tensor the Koszul complex above by $\O_{Y}(\Lambda_{m})$, the $j$th entry in the complex becomes $\sG_{j} = \O_{Y}(\Lambda_{m-j})^{\oplus {r \choose j}}$.

Note that by equations \autoref{eq.GlobalDivFujitaLocalVan} and \autoref{eq.GlobalDivFujitaGlobalVan} for $e \gg 0$, we know that
\[
R^{i} \pi_{*} \O_{Y}(\Lambda_{m-j}) = 0
\qquad \text{and} \qquad
H^{i}(Y, \O_{Y}(\Lambda_{m-j})) = 0
\]
for all $j = 0, \ldots, m$ and all $i > 0$.
This implies that the complex
\[
0 \to F^{e}_{*}\pi_{*} \sG_{r} \to F^{e}_{*}\pi_{*} \sG_{r-1} \to \cdots \to F^{e}_{*}\pi_{*} \sG_{1} \to F^{e}_{*}\pi_{*} \sG_{0} \to 0
\]
is exact, and furthermore that it remains exact after taking global sections.  As taking the image under $\Tr^{e}$ preserves surjectivity, we have once more a surjective map
\begin{equation}
\label{eq:finalsurjection}
\Tr^{e}\left( H^{0}(X, F^{e}_{*}\pi_{*}\sG_{1})\right) \to[( \; s_{1} \; s_{2} \; \cdots \; s_{r} \; )] \Tr^{e}\left( H^{0}(X, F^{e}_{*}\pi_{*}\sG_{0})\right)
\end{equation}
where our notation on the left indicates that the trace map has been applied individually to each direct summand of
\[
H^{0}(X, F^{e}_{*}\pi_{*}\sG_{1}) = \left[ H^{0}(X,F^{e}_{*}\pi_{*} \O_{Y} (\Lambda_{m-1}))\right]^{\oplus r}.
\]
We then have that both
\[
\sum_{e \gg 0} \Tr^{e}\left( H^{0}(X, F^{e}_{*}\pi_{*}\sG_{1})\right) = \left[P^{0}(X, \O_{X}(M + (m-1)L) \tensor_{\O_{X}} \tau(X, \ba^{m-1}))\right]^{\oplus r}
\]
and
\[
\sum_{e \gg 0} \Tr^{e}\left( H^{0}(X, F^{e}_{*}\pi_{*}\sG_{0})\right) = P^{0}(X, \O_{X}(M + mL) \tensor_{\O_{X}} \tau(X, \ba^{m}))
\]
hold, and the desired conclusion now follows immediately from the surjectivity of \eqref{eq:finalsurjection}.
\end{proof}

\section{Computations of $P^0$ for curves}

To conclude this article, we analyze $P^0$ in the case of curves.  Throughout this section, $C$ is a connected smooth projective curve over an algebraically closed field of characteristic $p > 0$.  Our first result shows that $P^{0}$ is often non-zero.

\begin{lemma}
If $\sL$ is a line bundle on $C$ of degree $\geq 2$, then the linear system $|P^0(C, \omega_C \tensor \sL)| \subseteq |H^0(C, \omega_C \tensor \sL)|$ is base point free.  If $\deg \sL \geq 3$, then $|P^0(C, \omega_C \tensor \sL)|$ induces a closed embedding.  In particular, in either case, $P^0(C, \omega_C \tensor \sL) \neq 0$.
\end{lemma}
\begin{proof}
For the first statement, suppose that $Q \in C$ is a point, fix an effective divisor $D$ on $C$.  Then writing $\sL = \O_C(L)$, we have the diagram of exact sequences for $e \gg 0$
\[
\scriptsize
\xymatrix@C=8pt{
H^0(C, F^e_* (\omega_C \tensor \O_C(p^e L - D))) \ar[r] \ar[d] &  H^0(Q, F^e_* (\omega_{p^e Q} \tensor \O_C(p^e L -D))) \ar[r] \ar[d] &  H^1(C, F^e_* (\omega_C \tensor \O_C(p^e (L -Q)-D))) \ar[d] \\
H^0(C, \omega_C \tensor \sL) \ar[r] & H^0(Q, \omega_Q \tensor \sL) \ar[r] & H^1(C, \omega_C \tensor \sL(-Q))
}
\]
where the vertical maps are induced by the trace.  Certainly the upper right term is zero by Serre vanishing since $e \gg 0$.  We need to prove that the middle vertical map is surjective (in other words, non-zero since $H^0(Q, \omega_Q \tensor \sL) = k$).  But consider now the maps on stalks:
\[
F^e_* \omega_{C,Q}(-D) \to \omega_{C,Q}
\]
induced by trace.  For $e \gg 0$, this is surjective since regular local rings are strongly $F$-regular \cite{HochsterHunekeTightClosureAndStrongFRegularity}.  But then the middle vertical map above is just obtained as a tensor product of the above map, and so it is surjective as well.  Thus we've obtained a non-vanishing section of $P^0(C, \omega_C \tensor \sL)$, as desired.

The remaining statement about $|P^0(C, \omega_C \tensor \sL)|$ inducing an embedding if $\deg \sL \geq 3$ follows similarly via separation of points and tangent vectors.
\end{proof}

Next, recall that H.~Tango proved in \cite[Theorem 15]{TangoOnTheBehaviorOfVBAndFrob} that $H^0(C, \omega_C \tensor \sL) = S^0(C, \omega_C \tensor \sL)$ for all $\sL$ of degree $\geq {2g - 2 \over p}$.
In the following proposition, we obtain a similarly explicit statement comparing $P^{0}$ and $H^{0}$ on curves, essentially giving an effective bound for $n \gg 0$ in \autoref{lem.PNotIsH0Smooth}.  This is an easy consequence of \cite[Lemma 3.2]{HaraDimensionTwo}

\begin{proposition}
\label{prop:tangoequiv}
If $\sL$ is a line bundle on  $C$ with $\deg \sL > {2g - 2 \over p}$, then
\[
\begin{array}{rl}
H^0(C, \omega_C \tensor \sL) = & P^0(C, \omega_C \tensor \sL) := P^0(C, \tau(C, -K_C) \tensor \sL).\\
\end{array}
\]
\end{proposition}
\begin{proof}
Fix any $D > 0$ on $C$.  For $e \gg 0$, the canonical map
\[
H^1(C, \sL^{-p^e} ) \to H^1(C, F_* \sL^{-p^{e+1}} \otimes \O_C(D))
\]
injects by \cite[Lemma 3.2]{HaraDimensionTwo}.  Pushing forward by $F^e_* $ and applying Serre duality we see that
\[
H^0(C, F^{e+1}_* \sL^{p^{e+1}} \otimes \O_C(-D)) \to H^0(C, F^e_* \sL^{p^e})
\]
surjects for any $e \gg 0$.  But $H^0(C, F^e_* \sL^{p^e}) \to H^0(C, \sL)$ surjects since $\deg \sL > {2 g - 2 \over p }$ by \cite{TangoOnTheBehaviorOfVBAndFrob}.  Combining these two surjections gives us our desired result.
\end{proof}

We immediately obtain the following corollary.

\begin{corollary}
Let $A$ be an ample divisor on $C$ and set $S = \bigoplus_{i \geq 0} H^0(C, \O_C(iA))$.  Then the $i$th graded components
\[
[\tau(\omega_S)]_i = [\omega_S]_i
\]
coincide for all $i$ such that $i \cdot \deg(A) > {2 g - 2 \over p}$.  In particular, taking $A = K_C$ if the genus of $C$ is at least two, we have $[\tau(\omega_S)]_i = [\omega_S]_i$ for all $i > 0$.
\end{corollary}
\begin{proof}
As argued previously, we have $P^0(C, \omega_C \tensor \O_C(iA)) = [\tau(\omega_S)]_i$ and $H^0(C, \omega_C \tensor \O_C(iA)) = [\omega_S]_i$, and so the corollary follows immediately from \autoref{prop:tangoequiv}.
\end{proof}

\bibliographystyle{skalpha}
\bibliography{MainBib}
\end{document}